\documentclass[11pt]{amsart}
\usepackage{amsmath,enumerate,amsfonts,amssymb,amsthm, verbatim}
\newtheorem{theorem}{Theorem}
\newtheorem{lemma}[theorem]{Lemma}
\newtheorem{proposition}[theorem]{Proposition}
\newtheorem{corollary}[theorem]{Corollary}

\theoremstyle{definition}
\newtheorem{definition}[theorem]{Definition}
\newtheorem{example}[theorem]{Example}

\newtheorem{remark}[theorem]{Remark}

\newcommand{\Section}[1]{\section{#1}\setcounter{theorem}{0}}

\newcommand{\tsum}{\textstyle\sum}
\newcommand{\<}{\langle}
\renewcommand{\>}{\rangle}
\newcommand{\N}{{\mathbb N}}
\newcommand{\Z}{{\mathbb Z}}

\renewcommand{\phi}{{\varphi}}
\newcommand{\R}{{\mathbb R}}
\newcommand{\C}{{\mathbb C}}
\renewcommand{\v}{{\mathfrak v}}
\newcommand{\z}{{\mathfrak z}}
\newcommand{\g}{{\mathfrak g}}
\newcommand{\so}{{\mathfrak{so}}}
\newcommand{\n}{{\mathfrak n}}
\newcommand{\m}{{\mathfrak m}}
\newcommand{\E}{{\mathcal{E}}}
\renewcommand{\O}{{\operatorname{O}}}
\newcommand{\liebr}{{[\,\,,\,]}}
\newcommand{\scp}{{\<\,\,,\,\>}}
\newcommand{\spann}{{\operatorname{span}}}
\newcommand{\kernn}{{\operatorname{ker}}}
\newcommand{\divv}{{\operatorname{div}}}
\newcommand{\grad}{{\operatorname{grad}}}
\newcommand{\ric}{{\operatorname{ric}}}
\newcommand{\Ric}{{\operatorname{Ric}}}
\newcommand{\rcirc}{{\operatorname{\overset{\hphantom{||}_\circ}{\textit R}}}}
\newcommand{\threestar}{{(*}{*}{*)}}

\newcommand{\scal}{{\operatorname{\textsl scal}}}
\newcommand{\Id}{{\operatorname{Id}}}
\newcommand{\Tr}{{\operatorname{Tr}}}
\renewcommand{\Pr}{{\operatorname{Pr}}}
\newcommand{\dimm}{{\operatorname{dim}}}
\newcommand{\dvol}{{\operatorname{\textit{dvol}}}}
\newcommand{\vol}{{\operatorname{vol}}}
\newcommand{\inv}{^{-1}}

\newcommand{\restr}[1]{\lower0.4ex\hbox{$|$}\lower0.7ex
  \hbox{$\scriptstyle{#1}$}}

\begin{document}
\title[Inaudibility of sixth order curvature invariants]{Inaudibility of sixth
order curvature invariants}
\author{Teresa Arias-Marco}
\address{Departamento de Matem\'aticas, Universidad de Extremadura, 06006 Badajoz,
Spain}
\email{ariasmarco@unex.es}
\author{Dorothee Schueth}
\address{Institut f\"ur Mathematik, Humboldt-Universit\"at zu
Berlin, D-10099 Berlin, Germany}
\email{schueth@math.hu-berlin.de}


\thanks{The authors were partially supported by by DFG
Sonderforschungsbereich~647. The first author's work has also been supported
by D.G.I.~(Spain) and FEDER Project MTM2013-46961-P, by Junta de Extremadura and FEDER funds.}

\begin{abstract}
It is known that the spectrum of the Laplace operator on functions of a closed Riemannian
manifold does not determine the integrals of the individual fourth order curvature
invariants $\scal^2$, $|\ric|^2$, $|R|^2$, which appear as summands in the second heat
invariant~$a_2$. We study the analogous question for the integrals of the sixth order curvature
invariants appearing as summands in~$a_3$. Our result is that none of them is
determined individually by the spectrum, which can be shown using various examples.
In particular, we prove that two isospectral nilmanifolds of Heisenberg type
with three-dimensional center are locally isometric if and only if they have the
same value of $|\nabla R|^2$. In contrast, any pair of isospectral nilmanifolds of
Heisenberg type with centers of dimension $r>3$ does not differ in any
curvature invariant of order six, actually not in any curvature
invariant of order smaller than $2r$. We also prove that this implies that for
any $k\in\N$, there exist locally homogeneous manifolds which are not curvature
equivalent but do not differ in any curvature invariant of order up to~$2k$.
\end{abstract}

\maketitle

\Section{Introduction}
\label{sec:intro}
\noindent
Let $(M,g)$ be a closed Riemannian manifold. The eigenvalue spectrum (with multiplicities)
of the associated Laplace operator $\Delta_g=-\divv_g\grad_g$ acting on smooth functions
is classically known to determine not only
the dimension and the volume of $(M,g)$ (by Weyl's asymptotic formula), but also the so-called heat invariants
$a_0(g), a_1(g), a_2(g),\ldots$. These are defined as the coefficients appearing in
Minakshisundaram-Pleijel's asymptotic expansion
$$\Tr\bigl(\exp(-t\Delta_g)\bigr)\;\sim\;(4\pi t)^{-\dimm M/2}
\tsum_{q=0}^\infty a_q(g)t^q\quad\text{ for }\,t\searrow 0.
$$
Here,
\begin{equation*}
  \begin{split}
  a_0(g)&=\vol(M,g),\\
  a_1(g)&=\tfrac16\textstyle\int_M\scal\,\dvol_g,\\
  a_2(g)&=\tfrac1{360}\textstyle\int_M(5\scal^2-2|\ric|^2+2|R|^2)\dvol_g,
  \end{split}
\end{equation*}
where $\scal$, $\ric$ and $R$ denote the scalar curvature, the Ricci tensor and the Riemannian
curvature tensor of $(M,g)$, respectively. In general, each $a_q(g)$ is known to be the integral
of some curvature invariant of order~$2q$ on~$(M,g)$; see, e.g.,~\cite{Gi}.

By definition, a curvature invariant is a polynomial in the coefficients of the Riemannian
curvature tensor~$R$ and its covariant derivatives $\nabla R$, $\nabla^2R$, \dots,
where the coefficients are taken with respect to some orthonormal basis of the tangent space
at the point under consideration, and the polynomial is required to be invariant under changes
of the orthonormal basis. Following the definitions, e.g., in~\cite{GV},
such an invariant is called an invariant of order~$k$ if it
is a sum of terms each of which involves a total of~$k$ derivatives of the metric tensor.
Each occurrence of~$R$ or any of its contractions involves two derivatives; each occurrence of~$\nabla$
adds one more derivative.  (See the proof of Proposition~\ref{prop:formofcurvinvs} below for a more
explicit description.) So, for example, $|\nabla R|^2=\<\nabla R,\nabla R\>$ is a curvature
invariant of order six.

It is well-known that each nonzero curvature invariant must be of even order,
and that bases for the space of curvature invariants of order two, resp.~four, are given by
$$\{\scal\},\text{ resp. }\{\scal^2,|\ric|^2,|R|^2,\Delta\scal\}.
$$
Note that $\int_M\Delta\scal=0$, but each of the remaining three elements of the
above basis of the space of curvature invariants of order four
does appear in the linear combination constituting the integrand of~$a_2(g)$.

Two closed Riemannian manifolds are called \textit{isospectral} if their Laplacians have the
same eigenvalue spectra, including multiplicities. A geometric property or quantity
associated with closed Riemannian manifolds is called \textit{audible} if
it is determined by the spectrum. By the above, each $a_q$ is audible; in particular,
$a_2(g)=a_2(g')$ for
any isospectral manifolds $(M,g)$, $(M',g')$. So the integral of
$5\scal^2-2|\ric|^2+2|R|^2$ must be the same for both manifolds.

This does not hold for the individual terms in this linear combination:
In~\cite{simplyconn}, the second author gave the first examples of isospectral manifolds
that showed that the integrals of~$\scal^2$ and~$|\ric|^2$
are inaudible; other examples in~\cite{habil}
showed the same for the integral of~$|R|^2$.

The aim of this paper is to prove similar results for sixth order
curvature invariants. Note the following formula for $a_3(g)$ which was proved by T.~Sakai in~\cite{Sa}:
\begin{equation}
\label{eq:3HeatInv}
  \begin{split}
       a_3(g)=\tfrac1{45360}\textstyle\int_M&\bigl(-142|\nabla\scal|^2-26|\nabla\ric|^2
       -7|\nabla R|^2+35\scal^3-42\scal|\ric|^2+42\scal|R|^2\\
       &-36\Tr(\Ric^3)+20(*)-8(**)+24\hat R\bigr)\dvol_g;
  \end{split}
\end{equation}
for the definition of the curvature invariants denoted here by $(*)$, $(**)$, $\hat R$
(and two more, $\rcirc$ and~$\threestar$), we refer to~\eqref{eq:6OrInv} in
Section~\ref{sec:prelims}.

It is already known that the integral of the individual term $|\nabla\scal|^2$ can
indeed differ in pairs of isospectral manifolds: C.~Gordon and Z.~Szabo constructed pairs
of isospectral closed manifolds one of which has constant scalar curvature, while the other has
nonconstant scalar curvature; see~\cite{GS}.

In this paper, we will show that for each of the individual summands in~\eqref{eq:3HeatInv},
there exist examples of isospectral manifolds differing in the integral of that curvature invariant.
The most interesting of these is arguably
$|\nabla R|^2$ which vanishes if and only if the manifold is locally symmetric.
Although we do not know of any example proving inaudibility of local symmetry,
we do show that the integral of $|\nabla R|^2$ is inaudible.

For a few of the sixth order curvature invariants, inaudibility will follow
already from known examples of isospectral manifolds. To study the remaining ones,
we will use a certain class of locally homogeneous manifolds, namely,
Riemannian two-step nilmanifolds. These are quotients of two-step nilpotent
Lie groups, endowed with a left invariant metric, by
cocompact discrete subgroups. By local homogeneity,
each curvature invariant is a constant function on such a manifold.
We develop some general insight into
the structure of the curvature invariants of Riemannian two-step nilmanifolds
(Proposition~\ref{prop:formofcurvinvs})
and give explicit formulas for the fourth and some of the sixth order curvature
invariants in this setting (Lemma~\ref{lem:lowerorder}, Lemma~\ref{lem:nablaric}).
For $|\nabla R|^2$, $\hat R$ and $\rcirc$ we give
only partially explicit formulas (Lemma~\ref{lem:nablar}).
These formulas will, however, be sufficient to show
inaudibility of $\int|\nabla R|^2$, $\int\hat R$ and~$\int\rcirc$
by using isospectral pairs of nilmanifolds of Heisenberg type.

The latter constitute a special class of Riemannian two-step
nilmanifolds and were introduced by A.~Kaplan; the very first example of isospectral,
locally nonisometric Riemannian manifolds found by C.~Gordon~\cite{Go} in 1993 was a pair of
nilmanifolds of Heisenberg type.
Within this class, we prove, in particular, the following results:

$\bullet$ For any pair of isospectral nilmanifolds of Heisenberg type
with three-dimensional centers of the underlying Lie groups, equality of the
value of (the constant function) $|\nabla R|^2$ on these manifolds is equivalent to local isometry;
the same holds for $\hat R$ and $\rcirc$ (Theorem~\ref{thm:nablar-equiv}).
Since isospectral, locally nonisometric pairs of this type exist, this
implies inaudibility of these curvature invariants.

$\bullet$
A pair of isospectral nilmanifolds of Heisenberg type
where the dimension of the centers of the underlying Lie groups is~$r$
can never be distinguished by the value of any curvature invariant of order $2q<2r$
(Theorem~\ref{thm:zgreater}).

$\bullet$
Two locally nonisometric nilmanifolds of Heisenberg type are never
curvature equivalent, meaning that there is no isometry of the associated
metric Lie algebras intertwining the Riemannian curvature tensors (Proposition~\ref{prop:notcurveq}).
In particular, for any $k\in\N$ there exist pairs of locally homogeneous
manifolds which are not curvature equivalent, but do not differ in any
curvature invariant up to order~$2k$ (Theorem~\ref{thm:notcurveqinspiteofinvars}).

\medskip
This paper is organized as follows:

In Section~\ref{sec:prelims}, we present some background information about
space of sixth order curvature invariants, introducing a commonly used
basis for this space and explaining certain
integral relations between the basis elements. We also observe that for some of
the basis elements, it already follows from known isospectral examples
that their integrals are not audible.

In Section~\ref{sec:two-step}, we review Riemannian two-step nilmanifolds,
a method from~\cite{GW:1997}
for obtaining isospectral pairs in this class, and some examples. In the
case of Heisenberg type nilmanifolds, we explain the general relation between
isospectral, locally
nonisometric examples and the existence of nonisomorphic
modules for the Clifford algebra associated with the centers (Remark~\ref{rem:heisexist}).

In Section~\ref{sec:curv}, we gain insight into the structure of the curvature
invariants in the general two-step nilpotent setting
(Proposition~\ref{prop:formofcurvinvs}), give formulas for the curvature invariants
of order two and four (Lemma~\ref{lem:lowerorder}), and also for several curvature
invariants of order six (Lemma~\ref{lem:nablaric}, Lemma~\ref{lem:nablar}).
Those proofs which involve somewhat lengthy calculations are deferred to
the Appendix. Applying the formulas, we prove inaudibility of
$\int\Tr(\Ric^3)$, $\int|\nabla\ric|^2$, $\int(*)$, $\int(**)$, $\int\threestar$
using the examples from Section~\ref{sec:two-step}.
As an aside, we also give an
example where the isospectral manifolds differ in $|\ric|^2$ and in~$|R|^2$;
although inaudibility of $\int|\ric|^2$ and $\int |R|^2$ was already known,
this is the first such example in the class of nilmanifolds.

In Section~\ref{sec:heis} we study the structure of curvature invariants
in the special class of Heisenberg type nilmanifolds. We prove
inaudibility of $\int|\nabla R|^2$, $\int\hat R$, $\int\rcirc$ and the other results
mentioned above (Theorem~\ref{thm:nablar-equiv}, Theorem~\ref{thm:zgreater}, Proposition~\ref{prop:notcurveq},
Theorem~\ref{thm:notcurveqinspiteofinvars}).

\Section{Preliminaries}
\label{sec:prelims}
\noindent
Let $(M,g)$ be a closed Riemannian manifold of dimension~$n$
with Levi-Civita connection $\nabla$. Let~$R$ be the associated
Riemannian curvature tensor; our sign
convention is such that
$$R(X,Y)=\nabla_{[X,Y]}-[\nabla_X,\nabla_Y].
$$
We denote by $\scal$, $\ric$, and~$\Ric$
the scalar curvature, the Ricci tensor, and the Ricci operator, respectively.

It is well-known that the space of curvature invariants of order~six has dimension~17
provided that $n\geq 6$ (see~\cite{GV}). A basis for this space (and still a generating
system in lower dimensions~$n$) is the
following, using index notation with respect to local orthonormal bases and
the Einstein summation convention:
\begin{equation}\label{eq:6OrInv}
  \begin{gathered}
  \scal^3,\;\scal|\ric|^2,\;\scal|R|^2,\; \Tr(\Ric^3),\;
  (*):=\ric_{ik}\ric_{jl}R_{ijkl},\; (**):=\ric_{ij}R_{ipqr}R_{jpqr},\;\\
  \hat R:=R_{ijkl}R_{klpq}R_{pqij},\; \rcirc:=R_{ikjl}R_{kplq}R_{piqj},\;
  |\nabla\scal|^2,\; |\nabla\ric|^2,\; |\nabla R|^2,\;\\
  \threestar:=\nabla_i\ric_{jk}\nabla_k\ric_{ij},\; \scal\,\Delta\scal,\; \Delta^2\scal,\;
  \langle\Delta\ric,\ric\rangle=-\ric_{ij}\nabla^2_{kk}\ric_{ij},\;\\
  \langle\nabla^2\scal,\ric\rangle=\left(\nabla^2_{ij}\scal\right)\ric_{ij},\;
  \langle\Delta R,R\rangle=-R_{ijkl}\nabla^2_{pp}R_{ijkl}.
  \end{gathered}
\end{equation}
The integrals of seven of the invariants in this basis either vanish
or can be expressed as a linear combination of integrals of certain others: First, note that
(with our sign convention for~$\Delta$)
\begin{equation}\label{eq:note}
  \begin{split}
  &\textstyle\int_M \Delta^2\scal=\int_M \langle\nabla\Delta\scal,\nabla 1\rangle=0,\\
  &\textstyle\int_M \scal\,\Delta\scal=\int_M|\nabla \scal|^2,\\
  &\textstyle\int_M \langle\Delta\ric,\ric\rangle=\int_M|\nabla \ric|^2,\\
  &\textstyle\int_M \langle\Delta R,R\rangle=\int_M|\nabla R|^2.
  \end{split}
\end{equation}
Three more relations are give by the following proposition:

\begin{proposition}\
\label{prop:integralrelations}

\begin{itemize}
\item[(i)]
$\textstyle\int_M \langle\nabla^2\scal,\ric\rangle=-\tfrac{1}{2}\int_M |\nabla\scal|^2$,
\item[(ii)]
$\textstyle\int_M \threestar=\int_M \bigl(\tfrac{1}{4}|\nabla\scal|^2-\Tr(\Ric^3)+(*)\bigr)$,
\item[(iii)]
$\textstyle\int_M \rcirc=\int_M \bigl(\tfrac14|\nabla\scal|^2
  -|\nabla\ric|^2+\tfrac14|\nabla R|^2
  -\Tr(\Ric^3)+(*)+\tfrac12(**)-\tfrac14\hat R\bigr)$.
\end{itemize}
\end{proposition}

\begin{proof}
From~\cite{GV}, formula (2.19) we have
$$\nabla^4_{ijij}\scal=\Delta^2\scal
+\tfrac{1}{2}|\nabla\scal|^2+\langle\nabla^2\scal,\ric\rangle.
$$
From this we derive~(i) by integrating and using the facts that
$\int_M \Delta^2\scal=0$ and, analogously, $\int_M\nabla^4_{ijij}\scal =0$.
For (ii), we first notice that
$$\textstyle\int_M (\nabla^2_{ij}\ric_{ik})\ric_{jk}
  =-\textstyle\int_M \langle\nabla_j\ric_{ik},
  \nabla_i \ric_{jk}\rangle=-\textstyle\int_M\threestar.
$$
Moreover, formula~(2.16) from~\cite{GV} says
$$\bigl(\nabla^2_{ij}\ric_{ik}\bigr)\ric_{jk}=\tfrac12
  \langle\nabla^2\scal,\ric\rangle+\Tr(\Ric^3)-(*).
$$
Therefore, we obtain~(ii) by integrating this on both sides and using~(i).
Finally, formula (2.20) from~\cite{GV} is
\begin{equation*}
\begin{split}
  \nabla^4_{ijki}\ric_{jk}={}&\tfrac12\Delta^2\scal+\tfrac12|\nabla\scal|^2-2|\nabla\ric|^2
  +2\langle\nabla^2\scal,\ric\rangle+\langle\Delta\ric,\ric\rangle+3\threestar\\
  &+2\Tr(\Ric^3)-2(*)+\tfrac{1}{4}\langle\Delta R,R\rangle+\tfrac12(**)-\rcirc-\tfrac14\hat R.
\end{split}
\end{equation*}
To obtain (iii), we first integrate this on both sides
and again use the facts that $\int_M \Delta^2\scal=0$ and $\int_M\nabla^4_{ijij}\scal=0$.
Then we use the two last equalities of \eqref{eq:note} as well as (i) and~(ii).
\end{proof}

On the other hand, note that each of the remaining ten curvature invariants does appear in
formula~\eqref{eq:3HeatInv} for the third heat invariant.
Now, for each of the ten expressions
\begin{equation}\label{eq:ten}
  \begin{gathered}
  \textstyle\int_M|\nabla \scal|^2,\;\textstyle\int_M |\nabla \ric|^2,\;
  \textstyle\int_M|\nabla R|^2,\;\textstyle\int_M\scal^3,\;\textstyle\int_M\scal |\ric|^2,
  \;\textstyle\int_M \scal |R|^2,\\
  \textstyle\int_M \Tr(\Ric^3),\;
  \textstyle\int_M (*),\;\textstyle\int_M (**),\;\textstyle\int_M\hat R
  \end{gathered}
\end{equation}
constituting~$a_3$ one can ask whether its integral is audible; i.e.,
whether it is determined by the spectrum of the Laplace
operator on functions. Since a choice of basis was involved, the analogous question might
of course be asked for any fixed linear combination other than that appearing in~\eqref{eq:3HeatInv},
such as, for example,
\begin{equation}\label{eq:plustwo}
\textstyle\int_M \threestar,\text{\ }\textstyle\int_M \rcirc
\end{equation}
from the left hand sides in Proposition~\ref{prop:integralrelations}.
The most interesting of the above invariants is the integral over
$|\nabla R|^2$: It is zero if and only if the metric is locally symmetric.
Although we do not know any examples showing that local symmetry itself is
inaudible, we will indeed prove that the value of $|\nabla R|^2$ is
inaudible. For sake of completeness, we
will prove that actually {\it none\/} of the twelve integrals just mentioned is
audible.

\begin{remark}\label{rem:overview}
For a few of these this is obvious already from known isospectral examples:\\
(i)
As already mentioned in the Introduction,
in~\cite{GS} a pair of isospectral closed manifolds was constructed with the property
that one of them had constant scalar curvature while the other did not; in particular,
$$\textstyle\int_M |\nabla\scal|^2\text{ is not audible.}
$$
(ii)
In~\cite{GGSWW}, continuous families of isospectral metrics were
constructed with the property that the maximal value of the
scalar curvature changes during the deformation. More specifically,
Example~8 of that paper gave a family of isospectral metrics $g(t)$, $t\in[0,\frac18]$
on $M=S^5\times(\R^2/\Z^2)$
whose volume element coincides with the standard one
and whose scalar curvature at $(x,z)\in S^5\times T^2$ depends
only on $x\in S^5$ and is equal to (using Proposition~6 of~\cite{GGSWW})
\begin{equation*}
  \begin{split}-\frac{13}2+5\cdot 4+\frac12\bigl(
  &(2-5t)x_1^2+x_2^2+(4+8t)x_3^2+4x_4^2+(10-3t)x_5^2+9x_6^2\\
  &+2\sqrt{5t-40t^2}x_1x_3
  -2\sqrt{15}tx_1x_5+2\sqrt{3t-24t^2}x_3x_5\bigr).
  \end{split}
\end{equation*}
The integral of the third power of this expression over $x\in S^5$
is a nonconstant function of~$t$. More precisely, this integral turns out to be a
polynomial in~$t$ with leading term $t^3\cdot(45A-135B+(90-720)C)$, where
$A:=\int_{S^5} x_i^6\,dx=\int_{S^5} x_1^6\,dx$, $B:=\int_{S^5}x_1^4x_2^2\,dx$,
$C:=\int_{S^5}x_1^2x_2^2x_3^2\,dx$; we have $B=3C$ and $A=15C$, so $45A-135B-630C=-360C\ne0$.
In particular,
$$\textstyle\int_M\scal^3\text{ is not audible.}
$$
(iii)
In~\cite{habil}, continuous families of left invariant isospectral
metrics~$g_t$ on certain compact Lie groups~$G$ were constructed. By homogeneity,
the functions $\scal(g_t)$, $|\ric|^2(g_t)$, $|R|^2(g_t)$ are constant on~$G$ for
each fixed~$t$.
Since $a_0(g_t)=\vol(g_t)$ is constant in~$t$,
it follows by considering $a_1(g_t)=\frac16\int_G\scal(g_t)$ that $\scal(g_t)$ is constant
in~$t$, too. However, as shown in~\cite{habil}, the term $|\ric|^2(g_t)$
is nonconstant in~$t$ in these examples; by considering $a_2(g_t)$ it follows that
$|R|^2(g_t)$ is nonconstant in~$t$, too. Hence, these examples
show that
$$\textstyle\int_M\scal|\ric|^2\text{ and }\textstyle\int_M\scal|R|^2\text{ are not audible.}
$$
(iv)
In the following, we will show the same for the remaining eight invariants from~\eqref{eq:ten}
and~\eqref{eq:plustwo}. For this, we will be
able to use isospectral pairs of {\it locally homogeneous\/} isospectral manifolds (more precisely,
pairs of isospectral, locally non-isometric two-step nilmanifolds). In this case, each curvature
invariant is a constant function on the manifold. Therefore, and since two isospectral manifolds
have the same volume, proving that the integral of a certain curvature
invariant is different for two given locally homogeneous
isospectral manifolds amounts to showing that they differ in the
(constant) value of the curvature invariant itself.
\end{remark}

\Section{Isospectral two-step nilmanifolds}
\label{sec:two-step}

\noindent
Let $\v:=\R^m$ and $\z:=\R^r$ be endowed with
the standard euclidean inner product.

\begin{definition}
\label{def:twostep}
With any given linear map $j:\z\ni Z\mapsto j_Z\in\so(\v)$,
we associate the following objects:
\begin{itemize}
\item[(i)]
The two-step nilpotent metric Lie algebra~$(\g(j),\scp)$ with
underlying vector space $\R^{m+r}=\v\oplus\z$,
endowed with the standard euclidean inner product~$\scp$,
and whose Lie bracket $\liebr^j$ is
defined by letting $\z$ be central, $[\v,\v]^j\subseteq\z$ and
$\<j_Z X,Y\>=\<Z,[X,Y]^j\>$ for all $X,Y\in\v$ and $Z\in\z$.
\item[(ii)]
The two-step simply connected nilpotent Lie group~$G(j)$
whose Lie algebra is $\g(j)$, and the left invariant
Riemannian metric~$g(j)$ on $G(j)$ which coincides with the given
inner product~$\scp$ on $\g(j)=T_e G(j)$. Note that the Lie group
exponential map $\exp^j:\g(j)\to G(j)$ is a diffeomorphism because
$G(j)$ is simply connected and nilpotent. Moreover, by the
Campbell-Baker-Hausdorff formula, $\exp^j(X,Z)\cdot
\exp^j(Y,W)=\exp^j(X+Y,Z+W+\frac12[X,Y]^j)$ for all $X,Y\in\v$ and
$Z,W\in\z$.
\item[(iii)]
The subset $\Gamma(j):=\exp^j(\Z^m\oplus\frac12\Z^r)$ of~$G(j)$.
If $j$ satisfies $[\Z^m,\Z^m]^j\subset\Z^r$
then the Campbell-Baker-Hausdorff formula implies that
$\Gamma(j)$ is a subgroup of~$G(j)$; moreover, this subgroup
is then discrete and cocompact.
\end{itemize}
\end{definition}

\begin{remark}
\label{rem:isometric}
(i)
Note that \emph{each} Riemannian two-step nilmanifold is locally
isometric to some $(G(j),g(j))$: In fact, each simply connected,
two-step nilpotent Lie group~$G$, endowed with a left invariant metric~$g$,
can be viewed as some $(G(j),g(j))$. Namely,
let~$\z$ be a linear subspace of the metric Lie algebra $(\g,g_e)$ associated with $(G,g)$
such that $[\g,\g]\subseteq\z\subseteq\z(\g)$,
let $\v$ be the orthogonal complement of~$\z$ w.r.t.~$g_e$, and define
$j:\z\to\so(\v)$ by $g(j_ZX,Y)=g(Z,[X,Y])$.

(ii)
As is well-known, $G(j)$ admits uniform discrete subgroups~$\Gamma$
if and only if there \emph{exists} a basis of $\g(j)$ such that
the corresponding structure constants of $\liebr^j$ are rational.
Even if this is a case, then $\Gamma(j)$ from Definition~\ref{def:twostep}(iii)
might not be a subgroup. We will use $\Gamma(j)$ in Proposition~\ref{prop:gw}
below and in explicit examples,
while allowing other~$\Gamma$ in general statements.

(iii)
The group $\O(\v)\times\O(\z)$ acts on the real
vector space of linear maps $j:\z\to\so(\v)$ by
$$((A,B)j)(Z)=Aj_{B\inv(Z)}A\inv.
$$
We call $j$ and $j'$ \textit{equivalent} if there exists
$(A,B)\in\O(\v)\times\O(\z)$ such that $j'=(A,B)j$.
In that case, $(A,B)$ provides a metric Lie algebra
isomorphism from $(\g(j),\scp)$ to $(\g(j'),\scp)$.
This condition is also necessary:
The metric Lie algebras $(\g(j),\scp)$ and $(\g(j'),\scp)$
are isomorphic if and only if $j$ and~$j'$ are equivalent
(see~\cite{GW:1997}). This, in turn, is equivalent to
$(G(j),g(j))$ and $(G(j'),g(j'))$ being isometric by a result
from~\cite{Wi} concerning nilpotent Lie groups.
Moreover, isometry of
$(G(j),g(j))$ and $(G(j'),g(j'))$ is
equivalent to local isometry of pairs of quotients $(\Gamma\backslash G(j),g(j))$,
$(\Gamma'\backslash G(j'),g(j'))$ of these groups
by any choice of discrete subgroups $\Gamma,\Gamma'$, provided the quotients are endowed
with the associated Riemannian quotient metrics. These quotient metrics
are again denoted $g(j)$, resp.~$g(j')$.
\end{remark}

\begin{definition}
\label{def:jLisosp}\
\begin{itemize}
\item[(i)]
Two linear maps $j,j':\z\to\so(\v)$ are called \emph{isospectral}
if for each $Z\in\z$, the maps $j_Z,j'_Z\in\so(\v)$ are similar,
that is, have the same eigenvalues (with multiplicities) in~$\C$.
Since each $j_Z$ is skew-symmetric, this condition is equivalent
to the following: For each $Z\in\z$ there exists $A_Z\in\O(\z)$
such that $j'_Z=A_Zj_ZA_Z\inv$. Note that $A_Z$ may depend on~$Z$.
\item[(ii)]
Two lattices in a euclidean vector space are called \emph{isospectral}
if the lengths of their elements, counted with multiplicities,
coincide.
\end{itemize}
\end{definition}

The following proposition is a specialized version
of a result from~\cite{GW:1997}; see~\cite{geod}, Remark~2.5(ii) for
an explanation about how to derive it from the original, more
general version.

\begin{proposition}[\cite{GW:1997} 3.2, 3.7, 3.8]\
\label{prop:gw}
Let $j,j':\z\to\so(\v)$ be isospectral. Assume that both
$[\Z^m,\Z^m]^j$ and $[\Z^m,\Z^m]^{j'}$
are contained in $\Z^r$. For each $Z\in\Z^r$ assume that the lattices $\kernn
(j_Z)\cap\Z^m$ and $\kernn(j'_Z)\cap \Z^m$ are isospectral.
Then the compact Riemannian manifolds $(\Gamma(j)\backslash G(j),g(j))$
and $(\Gamma(j')\backslash G(j'),g(j'))$ are isospectral for the Laplace
operator on functions.
\end{proposition}

\begin{example}
\label{ex:fourthree}
Let $m:=4$, $r:=3$, and for $Z=(c_1,c_2,c_3)\in\z=\R^3$ let
$j_Z$, resp.~$j'_Z$, be the endomorphism of $\v=\R^4$
given by the matrix
$$
\left(\begin{smallmatrix}
0&-2c_1&-2c_2&-2c_3\\2c_1&0&-c_3&c_2\\2c_2&c_3&0&-c_1\\2c_3&-c_2&c_1&0
\end{smallmatrix}\right),\quad\mathrm{resp.}\quad
\left(\begin{smallmatrix}
0&-c_1&-c_2&-c_3\\c_1&0&-2c_3&2c_2\\c_2&2c_3&0&-2c_1\\c_3&-2c_2&2c_1&0
\end{smallmatrix}\right),
$$
with respect to the standard basis of~$\R^4$.
This pair of maps $j,j'$ is a special case of an example from~\cite{GW:1997}.
The eigenvalues of both $j_Z$ and $j'_Z$ are $\{\pm i|Z|,\pm 2i|Z|\}$, each
with multiplicity one if $Z\ne0$; so $j$ and $j'$ are isospectral. Moreover, $\kernn(j_Z)=
\kernn(j'_Z)=\{0\}$ for $Z\ne0$. Therefore, all conditions from Proposition~\ref{prop:gw}
are satisfied and $(\Gamma(j)\backslash G(j),g(j))$, $(\Gamma(j')\backslash G(j'),g(j'))$
are isospectral. In Section~\ref{sec:curv} (see Corollary~\ref{cor:trric3}),
we will use this example to show inaudibility of
$$\textstyle\int_M\Tr(\Ric^3).
$$
\end{example}

\begin{example}
\label{ex:fivethree}
Let $m:=5$, $r:=3$, and for $Z=(c_1,c_2,c_3)\in\z=\R^3$ let
$j_Z$, resp.~$j'_Z$, be the endomorphism of $\v=\R^5$
given by the matrix
$$
\left(\begin{smallmatrix}
0&0&0&-c_3&c_2\\ 0&0&c_3&0&-c_1\\0&-c_3&0&0&0\\c_3&0&0&0&0\\-c_2&c_1&0&0&0
\end{smallmatrix}\right),\quad\mathrm{resp.}\quad
\left(\begin{smallmatrix}
0&-c_3&0&0&0\\c_3&0&0&0&0\\0&0&0&-c_3&c_2\\0&0&c_3&0&-c_1\\0&0&-c_2&c_1&0
\end{smallmatrix}\right),
$$
with respect to the standard basis of~$\R^5$.
In~\cite{geod}, it was shown that this pair of maps $j,j'$ satisfies
the conditions of Proposition~\ref{prop:gw}, so $(\Gamma(j)\backslash G(j),g(j))$
and $(\Gamma(j')\backslash G(j'),g(j'))$ is a pair of isospectral eight-dimensional
manifolds. This pair of manifolds was used in~\cite{geod} to demonstrate that integrability
of the geodesic flow is an inaudible property. In
Section~\ref{sec:curv} (see Proposition~\ref{prop:nablaric})
we will use it to prove inaudibility of
$$\textstyle\int_M|\nabla\ric|^2,\quad\textstyle\int_M(*),\quad\int_M(**),\text{\quad and }
\textstyle\int_M\threestar.
$$
\end{example}

\begin{example}
\label{ex:heis}
If $j,j':\z\to\so(\v)$ are both of {\it Heisenberg type}, that is,
if $j_Z^2=j_Z^{\prime\,2}=-|Z|^2\Id_\v$ for all $Z\in\z$, then $j$
and~$j'$ are obviously isospectral because the eigenvalues of
both of $j_Z$ and $j'_Z$
then are $\pm i|Z|$, each with multiplicity
$(\dimm\,\v)/2$. Moreover, $\kernn(j_Z)=\kernn(j'_Z)=\{0\}$ for
all $Z\ne0$. Therefore, if the matrix entries of each $j_{Z_\alpha}$
with respect to $\{X_1,\ldots,X_m\}$ are integer, then all
all conditions of Proposition~\ref{prop:gw}
are satisfied and $(\Gamma(j)\backslash G(j),g(j))$, $(\Gamma(j')\backslash
G(j'),g(j'))$ are isospectral. Note that it was such a pair of manifolds
which Gordon constructed in~\cite{Go} as the very first
example of isospectral, locally non-isometric manifolds; in the notation
of Remark~\ref{rem:heisexist} below, these were the ones associated
with $j=\rho^3_{(2,0)}$ and $j'=\rho^3_{(1,1)}$.

In Section~\ref{sec:heis} below we will use pairs of isospectral
nilmanifolds of Heisenberg type to prove inaudibility of
$$\textstyle\int_M|\nabla R|^2,\quad \textstyle\int_M \hat R,
\text{\quad and }\textstyle\int_M \rcirc.
$$
More precisely, we will show that for any pair $N=(\Gamma\backslash G(j),g(j))$,
$N':=(\Gamma'\backslash G(j'),g(j'))$ of isospectral nilmanifolds of
Heisenberg type we have the equivalences
\begin{equation}\label{eq:nablar-equiv}
  \textstyle\int_N|\nabla R|^2=\textstyle\int_{N'}|\nabla R|^2
  \Longleftrightarrow\textstyle\int_N \hat R=\textstyle\int_{N'}\hat R
  \Longleftrightarrow\textstyle\int_N\rcirc=\textstyle\int_{N'}\rcirc,
\end{equation}
and, in case $\dimm\z=3$, that each of these equalities
is equivalent to local isometry of $N$ and~$N'$ (see Theorem~\ref{thm:nablar-equiv}).
Since there do exist
locally nonisometric isospectral examples with $\dimm\z=3$, this
will prove the desired inaudibility statements.

On the other
hand, in case $\dimm\z>3$ we will show that the three equalities
from~\eqref{eq:nablar-equiv} are always true, regardless
whether $N$ and $N'$ are locally isometric or not. Even more, the
integral of \textit{each}
of the sixth order curvature invariants occurring in~$a_3$ will coincide
for isospectral pairs $N, N'$ if $\dimm\z>3$; actually, the same will
hold for any curvature invariant of order strictly
smaller than $2\dim\z$ (see Theorem~\ref{thm:zgreater}).
\end{example}

\begin{remark}
\label{rem:heisexist}
Locally nonisometric pairs of isospectral nilmanifolds of Heisenberg
type with $r$-dimensional center of the underlying Lie group
exist precisely for $r=\dimm\z\in\{3,7,11,15,\ldots\}$. More precisely:

(i)
By the condition $j_Z^2=-|Z|^2\Id_\v$, the map $j:\z\to\so(\v)$ extends to
a representation of the real Clifford algebra~$C_r$, turning~$\v$ into
a module over~$C_r$; the Clifford multiplication by~$Z$ is given by $j_Z:\v\to\v$.
Each such module decomposes into copies of simple modules;
see~\cite{LM}, p.~31. In~\cite{CD} it was proved that if $\m$ is a simple module
over~$C_r$, endowed with an inner product with respect to which the
Clifford multiplication with each $Z\in\R^r$
is skew-symmetric,
then there exists an orthonormal basis of~$\m$ with respect to which all matrix
entries of the Clifford multiplications with the elements $Z_1,\ldots,Z_r$
of our given orthonormal basis of~$\R^r$ are in $\{1,0,-1\}$.

For each $r\in\{3,7,11,15,\ldots\}$
there are exactly two simple real modules $\m^r_+$ and $\m^r_-$
over~$C_r$ up to isomorphism; see, e.g., \cite{LM}, p.~32.
For a given such~$r$, these two simple $C_r$-modules have the same dimension~$d_r$.
They can be distinguished by the action of $\omega_r:=Z_1\cdot\ldots\cdot
Z_r\in C_r$: After possibly switching names, $\omega_r$ acts on $\m^r_+$ as $\Id$ and on
$\m^r_-$ as $-\Id$. Moreover, replacing the Clifford multiplication of each $Z\in\R^r$
on $\m^r_+$ by its negative gives a module isomorphic to~$\m^r_-$.

It follows by the above result from~\cite{CD} that we can identify both $\m^r_+$ and
$\m^r_-$ with $\R^{d_r}$ in such a way that for both modules, the Clifford multiplications
with $Z_1,\ldots,Z_r$ have matrix entries
in $\{-1,0,1\}$ with respect to the standard basis of~$\R^{d^r}$.
For $(a,b)\in\N_0\times\N_0$ let $\rho^r_{(a,b)}$ denote the representation
of~$C_r$ on $\v:=(\R^{d_r})^{\oplus(a+b)}$ viewed as $(\m_+^r)^{\oplus a}\oplus(\m_-^r)^{\oplus b}$.

For any pair $(a,b)$, $(a',b')$ in $\N_0\times\N_0$ with $a+b=a'+b'$
but $\{a,b\}\ne\{a,b\}$, consider the maps $j,j':\R^r=\z\to\so(\v)=\so(m)$,
where $m:=(a+b)d_r$ and where
$j_Z:=\rho^r_{(a,b)}(Z)$, $j'_Z:=\rho^r_{(a',b')}(Z)$ for each $Z\in\z=\R^r\subset C_r$.

Then $j, j'$ is a pair of maps as in Example~\ref{ex:heis} and thus yields a pair of isospectral
nilmanifolds of Heisenberg type. Moreover, these are not locally
isometric. To see this, we show that $j$ and $j'$ are not equivalent in the
sense of Remark~\ref{rem:isometric}(iii):

First note that the products $j_{Z_1}\cdot\ldots\cdot j_{Z_r}=\rho_{(a,b)}(\omega_r)$
and $j'_{Z_1}\cdot\ldots\cdot j'_{Z_r}=\rho_{(a',b')}(\omega_r)$
are equal to $\Id$ on the respective $\m^r_+$ components
and to $-\Id$ on the $\m^r_-$ components of~$\v$.
In particular,
\begin{equation}
\label{eq:differenttraces}
(\Tr(j_{Z_1}\ldots j_{Z_r}))^2=((a-b)d_r)^2\ne((a'-b')d_r)^2
=(\Tr(j'_{Z_1}\ldots j'_{Z_r}))^2.
\end{equation}
On the other hand, suppose there were $A\in\O(\v)$, $B\in\O(\z)$ such that
$j'_Z=Aj_{B\inv Z}A\inv$ for all $Z\in\z$.
Note that $B\inv(Z_1)\cdot\ldots\cdot B\inv(Z_r)=\det(B\inv)\omega_r$
(see~\cite{LM}, p.~34). Thus, we would have
$j'_{Z_1}\ldots j'_{Z_r}=\det(B)\inv Aj_{Z_1}\ldots j_{Z_r}A\inv$,
contradicting~\eqref{eq:differenttraces} since $\det(B)\in\{\pm1\}$.

(ii) In the context of~(i), the metric Lie algebras associated
with $\rho^r_{(a,b)}$ and $\rho^r_{(b,a)}$ are isomorphic; an isomorphism
is obviously given by $\v\oplus\z\ni (X,Z)\mapsto(X,-Z)\in\v\oplus\z$.
In particular, $(G(j),g(j))$ and $(G(j'),g(j'))$ are isometric if $j=\rho^r_{(a,b)}$,
$j'=\rho^r_{(a',b')}$ and $\{a,b\}=\{a',b'\}$.

(iii)
Since each real module over~$C_r$ is decomposable into simple modules, it follows
that for $r\in\{3,7,11,15,\ldots\}$ each linear map $j:\z\to\so(\v)$ of Heisenberg type
must be equivalent in the sense of Remark~\ref{rem:isometric}(iii) to one of the
maps $\rho^r_{(a,b)}$ from~(i).
On the other hand, for $r\notin\{3,7,11,15,\ldots\}$, there exists only one simple
module over~$C_r$ up to isomorphism (see~\cite{LM}, p.~32).
Thus, in any pair of maps $j,j':\R^r\to\so(\v)$
of Heisenberg type with $r\notin\{3,7,11,\ldots\}$,
$j$ and $j'$ are equivalent and cannot yield locally
nonisometric nilmanifolds.
\end{remark}

\Section{Curvature invariants of two-step nilmanifolds}
\label{sec:curv}
\noindent
We use the notation from Definition~\ref{def:twostep}(i), (ii). We consider
a fixed linear map $j:\z\to\so(\v)$ and write, for simplicity,
$\liebr:=\liebr^j$. Let $\{X_1,\ldots,X_m\}$, resp.~$\{Z_1,\ldots,Z_r\}$,
denote an orthonormal basis of~$\v$, resp.~$\z$, and let $\nabla$, $R$, $\ric$ denote the
Levi-Civita connection, the curvature tensor, and the Ricci tensor
associated with the metric~$g(j)$. Recall our sign convention for~$R$ from
Section~\ref{sec:prelims}.

\begin{lemma}
\label{lem:firstformulas}
Let $J:=J(j):=\sum_{\alpha=1}^r j_{Z_\alpha}^2$. For $X,Y,U,V\in\v$ and $Z,W\in\z$ we have
\begin{itemize}
\item[(i)]
$\nabla_XY=\tfrac12[X,Y]=\sum_{\alpha=1}^r\<j_{Z_\alpha}X,Y\>Z_\alpha\in\z,
\quad \nabla_XZ=\nabla_ZX=-\tfrac12j_ZX\in\v, \quad \nabla_ZW=0$.
\item[(ii)]
$\<R(\n_1,\n_2)\n_3,\n_4\>=0$ whenever $\n_i\in\{\v,\z\}$, $i=1,\ldots,4$,
and either none or an odd number of the~$\n_i$ is $\v$. Moreover,
\begin{equation*}
  \begin{split}\<R(X,U)Y,V\>
  &=\tsum_{\alpha=1}^r(\tfrac14\<j_{Z_\alpha}U,Y\>\<j_{Z_\alpha}X,V\>
    -\tfrac14\<j_{Z_\alpha}X,Y\>\<j_{Z_\alpha}U,V\>\\
    &\quad\quad\quad\quad-\tfrac12\<j_{Z_\alpha}X,U\>\<j_{Z_\alpha}Y,V\>),\\
  \<R(X,Y)Z,W\>&=\<R(Z,W)X,Y\>=-\tfrac14\<[j_Z,j_W]X,Y\>,\\
  \<R(X,Z)Y,W\>&=\tfrac14\<j_WX,j_ZY\>=-\tfrac14\<j_Zj_WX,Y\>.
  \end{split}
\end{equation*}
\item[(iii)]
$\ric(X,Y)=\tfrac12\<JX,Y\>, \quad
\ric(X,Z)=0, \quad \ric(Z,W)=-\tfrac14\Tr(j_Zj_W)$.
\end{itemize}
\end{lemma}

\begin{proof}
In principle, all these formulas can be found in~\cite{Eb}. Alternatively,
(i) follows from the Koszul formula and the definitions.
From~(i), one easily derives the first and third statements of~(ii) and
\begin{equation*}
  \begin{split}\<-\nabla_X\nabla_U Y+\nabla_{\nabla_XU}Y,V\>&=
  \tfrac14\<j_{[U,Y]}X,V\>-\tfrac14\<j_{[X,U]}Y,V\>\\
  &=\tfrac14\tsum_{\alpha=1}^r(\<j_{Z_\alpha}U,Y\>\<j_{Z_\alpha}X,V\>
    -\tfrac14\<j_{Z_\alpha}X,U\>\<j_{Z_\alpha}Y,V\>),
\end{split}
\end{equation*}
from which the second statement of~(ii) follows by skew-symmetrization
w.r.t.~$X$ and~$U$. Moreover,
$$\<R(X,Z)Y,W\>=-\<\nabla_X\nabla_ZY,W\>=\tfrac14\<[X,j_ZY],W\>=\tfrac14\<j_WX,j_ZY\>.
$$
Part (iii) follows directly from (i) and~(ii) by taking traces and using
the skew-symmetry of~$j_{Z_\alpha}$.
\end{proof}

\begin{remark}\label{rem:bigj}
Let $j':\z\to\so(\v)$ be isospectral to~$j$.
\begin{itemize}
\item[(i)]
Since $j(Z)$ and $j'(Z)$ are similar by definition,
we have $\Tr(j_Z^2)=\Tr(j^{\prime\,2}_Z)$ for all $Z\in\z$.
Thus, by polarization,
\begin{equation}\label{eq:jzjw}
\Tr(j_Zj_W)=\Tr(j'_Zj'_W)\text{ for all }Z,W\in\z.
\end{equation}
\item[(ii)]
In particular, by Lemma~\ref{lem:firstformulas}(iii),
the Ricci operators associated with $g(j)$ and~$g(j')$
coincide on~$\z$. Therefore, $\Tr(\Ric(g(j))^3)$ and
$\Tr(\Ric(g(j'))^3)$ are equal if and only if
$\Tr(J^3)=\Tr(J^{\prime\,3})$, where $J':=\sum_{\alpha=1}^r j^{\prime\,2}_{Z_\alpha}$
is defined analogously as~$J$.
\end{itemize}
\end{remark}

\begin{corollary}\label{cor:trric3}
The two isospectral manifolds from Example~\ref{ex:fourthree} differ
in the value of $\Tr(\Ric^3)$.
\end{corollary}

\begin{proof}
Here $J$ and~$J'$ are diagonal
with diagonal entries $-12,-6,-6,-6$, resp.~$-3,-9,-9,-9$.
In particular, $\Tr(J^3)=-2376\ne-2214=\Tr(J^{\prime\,3})$.
The statement now follows from Remark~\ref{rem:bigj}(ii).
\end{proof}

\begin{definition}
\label{def:Ik}
Let $q\in\N$.
For each tuple $(k_1,\ldots,k_{2q})$ in $\{1,\ldots,q\}^{2q}$
which arises as a permutation of $(1,1,2,2,\ldots,q,q)$, i.e., which
contains each entry exactly twice, we define
the following polynomial invariants of~$j$ of order~$2q$:
$$I_{k_1\ldots k_\lambda|\ldots|k_\mu\ldots k_{2q}}(j)
  :=\sum\Tr(j_{Z_{\alpha_{k_1}}}\ldots j_{Z_{\alpha_{k_\lambda}}})
  \cdot\ldots\cdot\Tr(j_{Z_{\alpha_{k_\mu}}}\ldots j_{Z_{\alpha_{k_{2q}}}}),
$$
where the sum is taken according to the Einstein summation convention: For
each pair $k_i=k_j$ the sum runs over $\alpha_{k_i}$ once from $1$ to~$r$.
So the sum has exactly $r^q$ summands (and not $r^{2q}$).
We also write $I_{k_1\ldots k_\lambda|\ldots|k_\mu\ldots k_{2q}}$
for $I_{k_1\ldots k_\lambda|\ldots|k_\mu\ldots k_{2q}}(j)$ if
the context is clear. Moreover, we will usually replace the numbers $k_i$ by
other symbols; for example,
$I_{\alpha\beta\alpha\beta}:=I_{1212}$.
\end{definition}

With $J$ as defined in Lemma~\ref{lem:firstformulas},
we have for $q=1$:
$$I_{\alpha\alpha}=\tsum_{\alpha=1}^r\Tr(j_{Z_\alpha}^2)=\Tr(J);
$$
note that $I_{\alpha|\alpha}=0$ since $\Tr(j_{Z_\alpha})=0$ for each~$\alpha$.
For $q=2$, the nonvanishing invariants of the above form are exactly
\begin{equation*}
  \begin{split}
  I_{\alpha\alpha|\beta\beta}&=\tsum_{\alpha,\beta=1}^r\Tr(j_{Z_\alpha}^2)\Tr(j_{Z_\beta}^2)
    =(\Tr(J))^2,\\
  I_{\alpha\alpha\beta\beta}&=\tsum_{\alpha,\beta=1}^r\Tr(j_{Z_\alpha}^2j_{Z_\beta}^2)=\Tr(J^2),\\
  I_{\alpha\beta|\alpha\beta}&=\tsum_{\alpha,\beta=1}^r(\Tr(j_{Z_\alpha}j_{Z_\beta}))^2,\\
  I_{\alpha\beta\alpha\beta}&=\tsum_{\alpha,\beta=1}^r\Tr(j_{Z_\alpha}j_{Z_\beta}j_{Z_\alpha}j_{Z_\beta}).
  \end{split}
\end{equation*}
Some examples for $q=3$ (not a complete list):
\begin{equation*}
  \begin{split}
%  I_{\alpha\alpha\beta\beta\gamma\gamma}&=
%    \tsum_{\alpha,\beta,\gamma=1}^r\Tr(j_{Z_\alpha}^2j_{Z_\beta}^2j_{Z_\gamma}^2)=\Tr(J^3),\\
  I_{\alpha\alpha\beta\gamma\gamma\beta}&=\tsum_{\beta=1}^r\Tr(Jj_{Z_\beta}Jj_{Z_\beta}),\\
  I_{\alpha\alpha\beta\gamma\beta\gamma}&=\tsum_{\beta,\gamma=1}^r\Tr(Jj_{Z_\beta}
    j_{Z_\gamma}j_{Z_\beta}j_{Z_\gamma}),\\
  I_{\alpha\alpha\beta\gamma|\beta\gamma}&=\tsum_{\beta,\gamma=1}^r\Tr(Jj_{Z_\beta}j_{Z_\gamma})
    \Tr(j_{Z_\beta}j_{Z_\gamma}),\\
%  I_{\alpha\gamma\beta\gamma|\alpha\beta}&=\tsum_{\alpha,\beta,\gamma=1}^r\Tr(j_{Z_\alpha}j_{Z_\gamma}
%    j_{Z_\beta}j_{Z_\gamma})\Tr(j_{Z_\alpha}j_{Z_\beta}),\\
  I_{\alpha\gamma|\beta\gamma|\alpha\beta}&=\tsum_{\alpha,\beta,\gamma=1}^r
    \Tr(j_{Z_\alpha}j_{Z_\gamma})\Tr(j_{Z_\beta}j_{Z_\gamma})\Tr(j_{Z_\alpha}j_{Z_\beta}),\\
  I_{\alpha\beta\gamma|\alpha\beta\gamma}&=\tsum_{\alpha,\beta,\gamma=1}^r(\Tr(j_{Z_\alpha}j_{Z_\beta}j_{Z_\gamma}))^2.
\end{split}
\end{equation*}
Note that it follows from skew-symmetry of the $j_Z$ that
$\Tr(j_{Z_\beta}j_{Z_\alpha}j_{Z_\gamma})=-\Tr(j_{Z_\alpha}j_{Z_\beta}j_{Z_\gamma})$
and thus $I_{\alpha\beta\gamma|\beta\alpha\gamma}=-I_{\alpha\beta\gamma|\alpha\beta\gamma}$.
The invariant $I_{\alpha\beta\gamma|\alpha\beta\gamma}$
will play a crucial role in the Heisenberg type case (see Section~\ref{sec:heis}).

\begin{remark}
\label{rem:Iequiv}
If $j$ and $j'$ are equivalent in the sense of Remark~\ref{rem:isometric}(iii) then
it follows that $I_{k_1\ldots k_\lambda|\ldots|k_\mu\ldots k_{2q}}(j)
=I_{k_1\ldots k_\lambda|\ldots|k_\mu\ldots k_{2q}}(j')$ for each of the invariants
from Definition~\ref{def:Ik}.
\end{remark}

\begin{lemma}
\label{lem:lowerorder}
For the curvature invariants $\scal$ (of order two) and $\scal^2$, $|\ric|^2$, $|R|^2$
(of order four) we have:
\begin{itemize}
\item[(i)]
$\hphantom{\scal^2}\llap{$\scal$}=\tfrac14\Tr(J)=\tfrac14I_{\alpha\alpha}$
\item[(ii)]
$\scal^2=\tfrac1{16}(\Tr(J))^2=\tfrac1{16}I_{\alpha\alpha|\beta\beta}$
\item[(iii)]
$\hphantom{\scal^2}\llap{$|\ric|^2$}=\tfrac14\Tr(J^2)+\tfrac1{16}I_{\alpha\beta|\alpha\beta}
  =\tfrac14I_{\alpha\alpha\beta\beta}+\tfrac1{16}I_{\alpha\beta|\alpha\beta}$
\item[(iv)]
$\hphantom{\scal^2}\llap{$|R|^2$}=\frac12\Tr(J^2)+\frac38I_{\alpha\beta|\alpha\beta}+\frac18I_{\alpha\beta\alpha\beta}
  =\frac12I_{\alpha\alpha\beta\beta}+\frac38I_{\alpha\beta|\alpha\beta}
     +\frac18I_{\alpha\beta\alpha\beta}$
\end{itemize}
\end{lemma}

\begin{proof}
(i), (ii), and (iii) are very easy to prove using Lemma~\ref{lem:firstformulas}(ii). We defer the proof of~(iv)
to the Appendix.
\end{proof}

\begin{lemma}\label{lem:nablaric}
Let $(*)$, $(**)$ be as in~\eqref{eq:6OrInv}. Then we have
\begin{itemize}
\item[(i)]
$\hphantom{|\nabla\ric|^2}\llap{$(*)$}=\tfrac3{16}I_{\alpha\alpha\beta\gamma\gamma\beta}$
\item[(ii)]
$\hphantom{|\nabla\ric|^2}\llap{$(**)$}
  =\;\tfrac18I_{\alpha\alpha\beta\gamma\gamma\beta}+\tfrac18I_{\alpha\alpha\beta\gamma\beta\gamma}
    +\tfrac18I_{\alpha\alpha\beta\gamma|\beta\gamma}+\tfrac1{32}I_{\alpha\gamma\beta\gamma|\alpha\beta}$
\item[(iii)]
$|\nabla\ric|^2=-\tfrac14\Tr(J^3)+\tfrac18I_{\alpha\alpha\beta\gamma\gamma\beta}
  -\tfrac18I_{\alpha\alpha\beta\gamma|\beta\gamma}-\tfrac1{32}I_{\alpha\gamma|\beta\gamma|\alpha\beta}$\newline
$\hphantom{|\nabla\ric|^2}=-\tfrac14I_{\alpha\alpha\beta\beta\gamma\gamma}
  +\tfrac18I_{\alpha\alpha\beta\gamma\gamma\beta}
  -\tfrac18I_{\alpha\alpha\beta\gamma|\beta\gamma}-\tfrac1{32}I_{\alpha\gamma|\beta\gamma|\alpha\beta}$
\end{itemize}
\end{lemma}

We defer the proof of Lemma~\ref{lem:nablaric} to the Appendix.

\begin{proposition}
\label{prop:nablaric}
The two isospectral manifolds from Example~\ref{ex:fivethree} differ
in each of the values of $(*)$, $(**)$, $\threestar$, and $|\nabla\ric|^2$.
\end{proposition}

\begin{proof}
Here, $J$ and $J'$ are diagonal with entries $-2,-2,-1,-1,-2$, resp.~$-1,-1,-2,-2,-2$.
In particular, $\Tr(J^3)=\Tr(J^{\prime\,3})$.
By an easy computation, $\Tr(Jj_{Z_\beta}J_{Z_\beta})=-8$ for $\beta=1,2,3$,
and $\Tr(J'j'_{Z_1}Jj'_{Z_1})=\Tr(J'j'_{Z_2}J'j'_{Z_2})=-8$, but
$\Tr(J'j'_{Z_3}J'j'_{Z_3})=-10$. Therefore,
\begin{equation}
\label{eq:nablaricdiff}
I_{\alpha\alpha\beta\gamma\gamma\beta}(j)=-24\ne-26=I_{\alpha\alpha\beta\gamma\gamma\beta}(j');
\end{equation}
in particular, the values of $(*)$ are different for the two manifolds.
The same statement for~$\threestar$ now follows immediately from Proposition~\ref{prop:integralrelations}(ii)
and Remark~\ref{rem:bigj}(ii),
together with the fact that $\nabla\scal=0$ on both manifolds, and that $\Tr(J^3)=\Tr(J^{\prime\,3})$ (see above).

Since the term $I_{\alpha\alpha\beta\gamma\gamma\beta}$
also occurs in $(**)$, the statement about $(**)$ will follow
once we show that the two manifolds do not differ in any of the
remaining three summands of~$(**)$ from Lemma~\ref{lem:nablaric}(ii).
We here have $j_{Z_\beta}^4=-j_{Z_\beta}^2$ for $\beta=1,2,3$ and
$(j_{Z_\beta}j_{Z_\gamma})^2=0$ whenever $\beta\ne\gamma$; the same statements hold for~$j'$.
So $I_{\alpha\alpha\beta\gamma\beta\gamma}$ here happens to be $\Tr(-J^2)=-14=\Tr(-J^{\prime\,2})$ for both manifolds.
Also, $\Tr(Jj_{Z_\beta}j_{Z_\gamma})=0$ whenever $\beta\ne\gamma$, and the same for~$j'$;
so $I_{\alpha\alpha\beta\gamma|\beta\gamma}$ equals $\sum_{\beta=1}^3\Tr(Jj_{Z_\beta}^2)\Tr(j_{Z_\beta}^2)
=\Tr(J^2)\Tr(J)=14\cdot(-8)=\Tr(J^{\prime\,2})\Tr(J')$ for both manifolds.
Finally, note that $\Tr(j_{Z_\alpha}j_{Z_\beta})=0$ for $\alpha\ne\beta$, and the same for~$j'$. Thus, in
this example,
$I_{\alpha\gamma\beta\gamma|\alpha\beta}=\sum_{\alpha,\gamma=1}^3\Tr((j_{Z_\alpha}j_{Z_\gamma})^2)\Tr(j_{Z_\alpha}^2)
=\sum_{\alpha=1}^3\Tr(j_{Z_\alpha}^4)\Tr(j_{Z_\alpha}^2)=-2\cdot2-2\cdot2-4\cdot 4$, and the same for~$j'$.

The statement about $|\nabla\ric|^2$ now follows immediately: By~\eqref{eq:nablaricdiff},
the two manifolds differ
in the second summand of the formula from Lemma~\ref{lem:nablaric}, while the
remaining summands are the same for both; for the fourth summand, this follows either from the above
considerations or directly
from equation~\eqref{eq:jzjw}.
\end{proof}

\begin{remark}
\label{rem:r2}
As an aside, we will use the formulas from Lemma~\ref{lem:lowerorder}
to give an example of a pair of isospectral nilmanifolds differing
in the integrals of the fourth order curvature invariants
$|\ric|^2$ and $|R|^2$ (see Example~\ref{ex:sixtwo} below).
Although these are not the first examples
of isospectral manifolds with this property (see the Introduction),
they are the first such examples in the category of nilmanifolds.
Considering the heat invariants $a_0$, $a_1$, and $a_2$, note that a pair of isospectral,
locally homogeneous manifolds differs in $|\ric|^2$ if and only it differs in~$|R|^2$.
In the case of two-step nilmanifolds, it follows from Lemma~\ref{lem:lowerorder}(iii) and
Remark~\ref{rem:bigj}(ii)
that such a pair differs in $|\ric|^2$ if and only it differs in the value of
$\Tr(J^2)$. In Example~\ref{ex:fourthree}, we had $\Tr(J^3)\ne\Tr(J^{\prime\,3})$.
Nevertheless, the values of $\Tr(J^2)$ and $\Tr(J^{\prime\,2})$ happen to
coincide in that example, so we need a different one. The following
is related to an example from \cite{habil}, Proposition~3.6(ii)
(after replacing $j_{Z_2}(t)$ from that context by $3j_{Z_2}(t/3)-i\Id$,
evaluating at $t=0$, resp.~$t=2$, and identifying $\C^3$ with~$\R^6$).
\end{remark}

\begin{example}
\label{ex:sixtwo}
Let $m:=6$, $r:=2$, and for $Z=(c_1,c_2)\in\z=\R^2$ let
$j_Z$, resp.~$j'_Z$, be the endomorphism of $\v=\R^6$
given by the matrix
$$
\left(\begin{smallmatrix}
0&0&3c_2&c_1+c_2&0&0\\0&0&0&0&c_2&0\\-3c_2&0&0&0&0&-c_1+c_2\\
-c_1-c_2&0&0&0&0&3c_2\\0&-c_2&0&0&0&0\\0&0&c_1-c_2&-3c_2&0&0
\end{smallmatrix}\right),\quad\mathrm{resp.}\quad
\left(\begin{smallmatrix}
0&2c_2&c_2&c_1+c_2&0&0\\-2c_2&0&2c_2&0&c_2&0\\-c_2&-2c_2&0&0&0&-c_1+c_2\\
-c_1-c_2&0&0&0&2c_2&c_2\\0&-c_2&0&-2c_2&0&2c_2\\0&0&c_1-c_2&-c_2&-2c_2&0
\end{smallmatrix}\right),
$$
with respect to the standard basis of~$\R^6$.
The maps $j$ and~$j'$ are isospectral since
$j_{(c_1,c_2)}$ and $j'_{(c_1,c_2)}$ have the same characteristic polynomial
$\lambda^6+(2c_1^2+21c_2^2)\lambda^4+(c_1^2+9c_2^2)^2\lambda^2+c_2^2(c_1^2+8c_2^2)^2$.
Moreover, $\kernn(j_{(c_1,c_2)})=\kernn(j'_{(c_1,c_2)})=\{0\}$ if $c_2\ne0$;
for $c_2=0, c_1\ne0$ both kernels are $\spann\{X_2,X_5\}$. Therefore,
all conditions of Proposition~\ref{prop:gw} are satisfied and
$(\Gamma(j)\backslash G(j),g(j))$, $(\Gamma(j')\backslash G(j'),g(j'))$
are isospectral. A direct computation reveals
$\Tr(J^2)=630\ne 598=\Tr(J^{\prime\,2})$. By Remark~\ref{rem:r2}, this
implies that the two manifolds differ in the value of $|\ric|^2$, and
also in the value of~$|R|^2$.
\end{example}

Propositon~\ref{prop:formofcurvinvs} below concerns
the structure of curvature invariants of arbitrary order of two-step nilpotent
Lie groups with left invariant metrics. This description will enable
us to arrive at certain conclusions for higher order curvature
invariants in a special case (see Theorem~\ref{thm:zgreater}).
We first need the following observation:

\begin{remark}
\label{rem:formj}
Using Lemma~\ref{lem:firstformulas}(i),~(ii) repeatedly, one sees that
$\<(\nabla^p_{A_1,\ldots,A_p} R)(B,C)D,E\>$ with $A_1,\ldots,A_p,\allowbreak B,C,D,E
\in\{X_1,\ldots,X_m,Z_1,\ldots,Z_r\}$ is a linear combination of terms of order $p+2$ in~$j$
which are (if not zero) of the form
\begin{equation}
\label{eq:formj}
\<j_{Z_{\alpha_1}}\ldots j_{Z_{\alpha_i}}X_{\ell_1},X_{\ell_2}\>\cdot\ldots\cdot
\<j_{Z_{\alpha_j}}\ldots j_{Z_{\alpha_{p+2}}}X_{\ell_{2a-1}},X_{\ell_{2a}}\>.
\end{equation}
Moreover, the multiset $\{X_{\ell_1},\ldots,X_{\ell_{2a}},Z_{\alpha_1},\ldots,Z_{\alpha_{p+2}}\}$
of vectors occurring in~\eqref{eq:formj}
arises from the multiset $\{A_1,\ldots,A_p,B,C,D,E\}$ by possibly enlarging it by one or several pairs
of equal vectors from~$\{Z_1,\ldots,Z_r\}$; the vectors from~$\v$ are the same in both multisets.
In particular, $\<(\nabla^p_{A_1,\ldots,A_p} R)(B,C)D,E\>=0$ if the multiset
$\{A_1,\ldots,A_p,B,C,D,E\}$ contains an odd number of vectors from~$\v$.
\end{remark}

\begin{proposition}
\label{prop:formofcurvinvs}
Let $q\in\N$.
On a two-step nilpotent Lie group $G(j)$, endowed with the left invariant
metric~$g(j)$, each curvature invariant of order~$2q$ can be expressed
as a linear combination of polynomial invariants of~$j$ of the form
$I_{k_1\ldots k_\lambda|\ldots|k_\mu\ldots k_{2q}}$ as in Definition~\ref{def:Ik}.
\end{proposition}

\begin{proof}
According to~\cite{PTV}, p.~4646 (see also~\cite{BGM}, p.~75ff.), each
curvature invariant of
order~$2q$ is a linear combination of certain Weyl invariants of the form
\begin{equation}
\label{eq:W}
W=\Tr_\sigma (\nabla^{p_1}R\otimes\ldots\otimes\nabla^{p_\nu}R),
\end{equation}
where $\nu\in\N$, $p_i\in\N_0$ for each $i\in\{1,\ldots,\nu\}$, $p_1+\ldots+p_\nu$ is even,
$2q=2\nu+p_1+\ldots+p_\nu$, $\sigma\in S_{2N}$, $2N=4\nu+p_1+\ldots+p_\nu$,
and $\Tr_\sigma$ denotes the complete trace with respect to~$\sigma$. The latter is defined as the
sum according to the Einstein summation convention
with respect to equal indices $k_i=k_j$ in the expression
$$(\nabla^{p_1}R\otimes\ldots\otimes\nabla^{p_\nu}R)(e_{s_{k_1}},\ldots,e_{s_{k_{2N}}})
=(\nabla^{p_1}R)(e_{s_{k_1}},\ldots,e_{s_{k_{p_1+4}}})\ldots(\nabla^{p_\nu}R)
(e_{s_{k_{2N-p_\nu-3}}},\ldots,e_{s_{k_{2N}}}),
$$
where $\{e_1,\ldots,e_n\}$ is an orthonormal
basis of the tangent space at the point under consideration and $(k_1,\ldots,k_{2N})$ arises
from $(1,1,2,2,3,3,\ldots,N,N)$ by the permutation~$\sigma$.

In our case, by Remark~\ref{rem:formj}, each summand of~$W$ in~\eqref{eq:W}
is a linear combination of products of terms as in~\eqref{eq:formj},
so $W$ itself is a linear combination of terms of the form
\begin{equation}
\label{eq:Wsummands}
\<j_{Z_{\alpha_{s_1}}}\ldots j_{Z_{\alpha_{s_c}}}X_{\ell_{u_1}},X_{\ell_{u_2}}\>\cdot\ldots\cdot
\<j_{Z_{\alpha_{s_d}}}\ldots j_{Z_{\alpha_{s_{2q}}}}X_{\ell_{u_{2a-1}}},X_{\ell_{u_{2a}}}\>,
\end{equation}
with each $s_i$ and each $u_j$ occurring exactly twice.
Summation over pairs of equal $u_j$ will transform~\eqref{eq:Wsummands}
into a term of the form
$\Tr(j_{Z_{\alpha_{k_1}}}\ldots j_{Z_{\alpha_{k_\lambda}}})\cdot\ldots\cdot
\Tr(j_{Z_{\alpha_{k_\mu}}}\ldots j_{Z_{\alpha_{k_{2q}}}})$ in which still each~$k_i$ occurs exactly twice;
summation over pairs of equal $k_i$ then yields $I_{k_1\ldots k_\lambda|\ldots|k_\mu\ldots k_{2q}}$.
\end{proof}

We conclude this section by giving some partial results for $|\nabla R|^2$, $\hat R$, $\rcirc$ which
we will use in Section~\ref{sec:heis} to prove their inaudibility:

\begin{lemma}\
\label{lem:nablar}
\begin{itemize}
\item[(i)]
$|\nabla R|^2=-\hphantom{\frac{17}{64}}\llap{$\frac32$}I_{\alpha\beta\gamma|\alpha\beta\gamma}+L_1$,
\item[(ii)]
$\hphantom{|\nabla R|^2}\llap{$\hat R$}=-\hphantom{\frac{17}{64}}\llap{$\frac7{16}$}
I_{\alpha\beta\gamma|\alpha\beta\gamma}+L_2$,
\item[(iii)]
$\hphantom{|\nabla R|^2}\llap{$\rcirc$}=-\frac{17}{64}
I_{\alpha\beta\gamma|\alpha\beta\gamma}+L_3$,
\end{itemize}
where $L_1, L_2, L_3$ are universal linear combinations of certain other
$I_{k_1\ldots k_\lambda|\ldots|k_\mu\ldots k_6}$ in which all occurring subtuples
$(k_1,\ldots,k_\lambda)$, \dots, $(k_\mu,\ldots,k_6)$
are of even length.
\end{lemma}

We defer the proof of Lemma~\ref{lem:nablar} to the Appendix.

\Section{Curvature invariants of Heisenberg type nilmanifolds}
\label{sec:heis}
\noindent
We continue to use the notation from Definition~\ref{def:twostep}(i), (ii),
and we now always consider linear maps $j:\z\to\so(\v)$
of Heisenberg type. Recall from Example~\ref{ex:heis} that this means
$j_Z^2=-|Z|^2\Id_\v$ for all $Z\in\z$. By polarization, this is equivalent to
\begin{equation}
\label{eq:heis}
j_Zj_W+j_Wj_Z=-2\<Z,W\>\Id_\v\text{ for all }Z,W\in\z.
\end{equation}
Again, let $\{X_1,\ldots,X_m\}$ and $\{Z_1,\ldots,Z_r\}$
be orthonormal bases of $\v$ and~$\z$, respectively.

\begin{lemma}
\label{lem:heis}
In the Heisenberg type case, the following holds:
\begin{itemize}
\item[(i)]
$j_Zj_W=-j_Wj_Z$ for all $Z,W\in\z$ with $Z\perp W$.
\item[(ii)]
Let $k\in\N$ and $(\alpha_1,\ldots,\alpha_k)\in\{1,\ldots,r\}^k$. Let $\ell\in\{0,\ldots,k\}$
and $\beta_1<\ldots<\beta_\ell$
be such that $\{\beta_1,\ldots,\beta_\ell\}$ consists precisely of those $\alpha_i$ which
occur an odd number of times in $(\alpha_1,\ldots,\alpha_k)$.
Then there exists $c\in\{0,1\}$, depending only on the tuple $(\alpha_1,\ldots,\alpha_k)$,
but not on~$j$,
such that
$$j_{Z_{\alpha_1}}\ldots j_{Z_{\alpha_k}}=(-1)^c\,j_{Z_{\beta_1}}\ldots j_{Z_{\beta_\ell}},
$$
where in case $\ell=0$, the empty product $j_{Z_{\beta_1}}\ldots j_{Z_{\beta_\ell}}$
is to be read as $\Id_\v$.
\item[(iii)]
If $\ell$ is a positive even number and $\beta_1,\ldots,\beta_\ell\in\{1,\ldots,r\}$ are pairwise different
then $\Tr(j_{Z_{\beta_1}}\ldots j_{Z_{\beta_\ell}})=0$.
\item[(iv)]
If $\ell$ is positive, but strictly smaller than~$r$, then $\Tr(j_{Z_{\beta_1}}\ldots
j_{Z_{\beta_\ell}})=0$
for all $\beta_1,\ldots,\beta_\ell\in\{1,\ldots,r\}$. Trivially, the same
holds if $\ell=1$.
\end{itemize}
\end{lemma}

\begin{proof}
Part (i) is trivial by~\eqref{eq:heis}.
For (ii), one first repeatedly uses~(i) to arrange the factors in nondecreasing order w.r.t.~the
values of the~$\alpha_i$; the statement then follows from $j_{Z_{\alpha_i}}^2=-\Id_\v$.
If $\ell$ is positive and even, and $\beta_1,\ldots,\beta_\ell$ are pairwise different,
then (i) and the cyclicity of the trace imply $\Tr(j_{Z_{\beta_1}}\ldots j_{Z_{\beta_\ell}})
=-\Tr(j_{Z_{\beta_\ell}} j_{Z_{\beta_1}}\ldots j_{Z_{\beta_{\ell-1}}})=-\Tr(j_{Z_1}\ldots j_{Z_{\beta_\ell}})$,
hence~(iii).

For proving~(iv), it now
suffices to consider the case that $\ell$ is odd.
Since $\ell<r$, we can choose $\alpha\in\{1,\ldots,r\}\setminus\{\beta_1,\ldots,\beta_\ell\}$.
Then, using $j_{Z_\alpha}\inv=-j_{Z_\alpha}$ and~(i), we have
$\Tr(j_{Z_{\beta_1}}\ldots j_{Z_{\beta_\ell}})=\Tr(j_{Z_\alpha}j_{Z_{\beta_1}}\ldots
j_{Z_{\beta_\ell}}(-j_{Z_\alpha}))=\Tr(j_{Z_\alpha}^2j_{Z_{\beta_1}}\ldots j_{Z_{\beta_\ell}})
=-\Tr(j_{Z_{\beta_1}}\ldots j_{Z_{\beta_\ell}})$, hence~(iv).
\end{proof}

\begin{corollary}
\label{cor:heisI}
In the Heisenberg type case, the following holds:
\begin{itemize}
\item[(i)]
Any $I_{k_1\ldots k_\lambda|\ldots|k_\mu\ldots k_{2q}}$ as in Definition~\ref{def:Ik}
in which all the
occurring subtuples $(k_1,\ldots,k_\lambda)$, \dots, $(k_\mu,\ldots,k_{2q})$ are
of even length can be expressed
as a universal polynomial in $m=\dimm\v$ and $r=\dimm\z$ which does not
depend on~$j$.
\item[(ii)]
If at least one of the subtuples of odd length occurring in
$I_{k_1\ldots k_\lambda|\ldots|k_\mu\ldots k_{2q}}$
becomes strictly shorter than~$r$ or equal to one
after eliminating pairs of equal indices $k_i=k_j$ within that subtuple, then
$I_{k_1\ldots k_\lambda|\ldots|k_\mu\ldots k_{2q}}=0$.
\end{itemize}
\end{corollary}

\begin{proof}
Let $d$ be the length of one of the subtuples,
and let $\Tr(j_{Z_{\alpha_1}}\ldots j_{Z_{\alpha_d}})$ be the corresponding
factor in one of the $r^q$ summands occurring in the sum as which
$I_{k_1\ldots k_\lambda|\ldots|k_\mu\ldots k_{2q}}$ is defined.

By Lemma~\ref{lem:heis}(ii), $\Tr(j_{Z_{\alpha_1}}\ldots j_{Z_{\alpha_d}})$
can be simplified
to either $\pm\Tr(\Id_\v)\allowbreak=\pm m$ (where the sign does not depend on~$j$)
or to a new term which involves
only pairwise different $Z_{\alpha_i}$ and whose length $d'\le d$ is positive
and has the same parity as~$d$.

In this latter case, if $d$ and hence $d'$ is even, then the new term vanishes
by Lemma~\ref{lem:heis}(iii). This proves part~(i). If $d$ is odd, then the condition of~(ii)
implies, a forteriori, that $d'<r$ or $d'=1$ (note that there might be even more
equal indices $\alpha_i$ in $(\alpha_1,\ldots,\alpha_d)$ than equal indices~$k_i$
in the corresponding subtuple of $I_{k_1\ldots k_\lambda|\ldots|k_\mu\ldots k_{2q}}$).
So in this case, the new term vanishes by Lemma~\ref{lem:heis}(iv). This proves part~(ii).
\end{proof}

\begin{proposition}
\label{prop:heisequal}
\begin{itemize}
\item[(i)]
In the Heisenberg type case,
each curvature invariant of order two or four and each of
$\Tr(\Ric^3)$, $(*)$, $(**)$, $\threestar$, $|\nabla\ric|^2$
can be expressed
as a universal polynomial in $m=\dimm\v$ and $r=\dimm\z$ which does not
depend on~$j$.
\item[(ii)]
Any two isospectral nilmanifolds of Heisenberg type do not differ in any of
the curvature invariants mentioned in~(i).
\end{itemize}
\end{proposition}

\begin{proof}
For (i), just observe using Lemma~\ref{lem:lowerorder} and Lemma~\ref{lem:nablaric}
that each of these curvature invariants is a universal linear combination
of terms satisfying the condition of Corollary~\ref{cor:heisI}(i).
Part (ii) follows from (i) and Remark~\ref{rem:vz} below.
\end{proof}

\begin{remark}
\label{rem:vz}
Any two isospectral nilmanifolds of Heisenberg
type share the same dimensions $m=\dimm\v$ and also the same dimensions
$r=\dimm\z$.

To see this, let $N$ and $N'$ be two isospectral nilmanifolds of Heisenberg type,
associated with $j:\R^r\to\so(\R^m)$ and $j':\R^{r'}\to\so(\R^{m'})$, respectively.
Then necessarily $m+r=m'+r'$ since the dimension is spectrally determined.
Moverover, the two manifolds must have the same volume and the same total scalar
curvature, thus
$\scal(g(j))=\scal(g(j'))$. By Lemma~\ref{lem:lowerorder}(i) this means
$\Tr(J)=\Tr(J')$; hence $-mr=-m'r'$.
Together with $m+r=m'+r'$ this implies $\{m,r\}=\{m',r'\}$.
Using the classification of nilmanifolds of Heisenberg type from~\cite{BTV},
or recalling from Remark~\ref{rem:heisexist} that $\R^m$ is a module over $C_r$
and inspecting the dimensions of the simple real modules over $C_r$ in~\cite{LM},
one sees $m>r$ and $m'>r'$. So indeed we have $m=m'$ and $r=r'$.
\end{remark}

\begin{proposition}
\label{prop:heisImore}
Let $j,j':\z=\R^r\to\so(\v)=\so(m)$ be of Heisenberg type.
\begin{itemize}
\item[(i)]
If $2q<2r$ then $I_{k_1\ldots k_\lambda|\ldots|k_\mu\ldots k_{2q}}(j)
=I_{k_1\ldots k_\lambda|\ldots|k_\mu\ldots k_{2q}}(j')$ for each of the invariants
from Definition~\ref{def:Ik}.
\item[(ii)]
In the case $2q=2r$, the only invariants from Definition~\ref{def:Ik}
in which $j$ and $j'$ can possibly differ
are the $I_{k_1\ldots k_r|k_{\tau(1)}\ldots k_{\tau(r)}}$, where $\tau\in S_r$.
Note that $I_{k_1\ldots k_r|k_{\tau(1)}\ldots k_{\tau(r)}}=\pm
I_{k_1\ldots k_r|k_1\ldots k_r}$ due to Lemma~\ref{lem:heis}(i), depending
on the sign of the permutation~$\tau$.
\item[(iii)]
$j$ and $j'$ are equivalent in the sense of Remark~\ref{rem:isometric}(iii)
if and only if $I_{k_1\ldots k_r|k_1\ldots k_r}(j)=I_{k_1\ldots k_r|k_1\ldots k_r}(j')$.
\end{itemize}
\end{proposition}

\begin{proof}
By Corollary~\ref{cor:heisI}, $j$ and $j'$
cannot differ in $I_{k_1\ldots k_\lambda|\ldots|k_\mu\ldots k_{2q}}$
unless at least one of the subtuples $(k_1,\ldots,k_\lambda)$, \dots, $(k_\mu,\ldots,k_{2q})$
is of odd length at least~$\dimm\z=r$, after elimininating any pairs of equal indices
occurring within that subtuple. Each of the remaining (at least~$r$) indices has to
occur in one of the other subtuples (recall that each $k_i$ occurs exactly twice
in $(k_1,\ldots,k_{2q})$). But this implies $2q\ge r+r$ and, in the case $2q=2r$, that
there are exactly two subtuples, both of length~$r$. This shows (i) and~(ii).

The ``only if'' statement of~(iii) is a special case of Remark~\ref{rem:Iequiv}.
For the converse, let $j$ and~$j'$ be nonequivalent. By Remark~\ref{rem:heisexist}(ii),~(iii),
it follows that $r\in\{3,7,11,15,\ldots\}$, and that $j,j'$ are equivalent to
certain $\rho^r_{(a,b)}$, resp.~$\rho^r_{(a',b')}$ with $a+b=m=a'+b'$, but $\{a,b\}\ne\{a',b'\}$;
in particular, $|a-b|\ne|a'-b'|$.
Note that since $r$ is odd,
$\Tr(j_{Z_{\alpha_1}}\ldots j_{Z_{\alpha_r}})=0$ whenever $\alpha_1,\ldots,\alpha_r$
are not pairwise distinct (recall Lemma~\ref{lem:heis}(ii) and~(iv)). Moreover,
$(\Tr(j_{Z_{\alpha_1}}\ldots j_{Z_{\alpha_r}}))^2$ does not change under
permutations of $\alpha_1,\ldots,\alpha_r$ due to Lemma~\ref{lem:heis}(i).
So $I_{k_1\ldots k_r|k_1\ldots k_r}(j)=r!(\Tr(j_{Z_1}\ldots j_{Z_r}))^2$,
and similarly for~$j'$.
Now $I_{k_1\ldots k_r|k_1\ldots k_r}(j)\ne I_{k_1\ldots k_r|k_1\ldots k_r}(j')$
follows by~\eqref{eq:differenttraces} from Remark~\ref{rem:heisexist}(i).
\end{proof}

\begin{theorem}
\label{thm:zgreater}
\begin{itemize}
\item[(i)]
Any two isospectral nilmanifolds of Heisenberg type with $\dimm\z=r$ cannot
differ in any curvature invariant of order~$2q<2r$.
\item[(ii)]
Any two isospectral nilmanifolds of Heisenberg type with centers of dimension
strictly greater than three ($r>3$) do not differ
in any of the sixth, eighth, tenth or twelfth order curvature invariants.
\end{itemize}
\end{theorem}

\begin{proof}
(i)
By Proposition~\ref{prop:formofcurvinvs}, each curvature invariant of order~$2q$ is a
linear combination (with universal coefficients) of certain $I_{k_1\ldots k_a|\ldots|k_b\ldots k_{2q}}$.
Thus, the statement follows immediately from Remark~\ref{rem:vz} and Proposition~\ref{prop:heisImore}(i).

(ii) For the sixth order curvature invariants, this follows directly from (i).
The statement for eighth, tenth and twelfth order curvature invariants equally follows from~(i) after
recalling from Remark~\ref{rem:heisexist} that any two isospectral
nilmanifolds of Heisenberg type with $r>3$ are either locally isometric (and the statement
thus trivial) or satisfy $r\in\{7,11,15,\dots\}$, thus $r\ge7$.
\end{proof}

\begin{theorem}
\label{thm:nablar-equiv}
Let $N$, $N'$ be two isospectral nilmanifolds of Heisenberg type
associated with Lie algebras satisfying $r=\dimm\z$.
If $r=3$ then the following conditions are equivalent:
\begin{itemize}
\item[(a)]
$N$ and $N'$ are locally isometric.
\item[(b)]
$N$ and $N'$ have the same value of $|\nabla R|^2$.
\item[(c)]
$N$ and $N'$ have the same value of $\hat R$.
\item[(d)]
$N$ and $N'$ have the same value of $\rcirc$.
\end{itemize}
If $r\ne3$, then (b), (c), (d) are true regardless of~(a).
\end{theorem}

\begin{proof}
Trivially, (a) implies each of the other three statements. Moreover,
if $r\notin\{3,7,11,15,\ldots\}$ then (b), (c), (d) are true by Remark~\ref{rem:heisexist}(iii).
Let $\g(j)$, $\g(j')$ be the metric Lie algebras associated with $N$,~$N'$.
By Lemma~\ref{lem:nablar} and Corollary~\ref{cor:heisI}(i), each of (b), (c), (d)
is equivalent to
\begin{equation}
\label{eq:abg}
I_{\alpha\beta\gamma|\alpha\beta\gamma}(j)
=I_{\alpha\beta\gamma|\alpha\beta\gamma}(j').
\end{equation}
For $r>3$, this is always true by Theorem~\ref{thm:zgreater}(i).
For $r=3$, \eqref{eq:abg} is equivalent to~(a) by
Proposition~\ref{prop:heisImore}(iii)
and Remark~\ref{rem:isometric}(iii).
\end{proof}

\begin{corollary}
\label{cor:nablar}
In any pair of isospectral, locally nonisometric manifolds of Heisenberg type associated
with Lie algebras satisfying $r=\dimm\z=3$, the two manifolds differ in each of the values of $|\nabla R|^2$,
$\hat R$, $\rcirc$. Since such pairs do exist (see Remark~\ref{rem:heisexist}(i)), neither
$\int |\nabla R|^2$ nor $\int\hat R$ nor $\int\rcirc$ is audible.
\end{corollary}

Two locally homogeneous manifolds $(M,g)$, $(M',g')$ are called \textit{curvature equivalent} (of order zero)
if for $p\in M$, $p'\in M'$ there exists an euclidean isometry $F:(T_pM,g_p)\to (T_{p'}M',g'_{p'})$
which intertwines the Riemannian curvature operators; that is, $F(R(X,Y)Z)=R(F(X),F(Y))F(Z)$
for all $X,Y,Z\in T_pM$. The following result provides a certain contrast to Theorem~\ref{thm:zgreater}(i):

\begin{proposition}
\label{prop:notcurveq}
Let $N$ and $N'$ be any two nilmanifolds of Heisenberg type (without restriction to the dimensions
of the centers). If $N$ and $N'$ are not
locally isometric, then they are not curvature equivalent.
\end{proposition}

For the proof, the following lemma will serve as the key:

\begin{lemma}
\label{lem:TrRvq}
Let $j:\z=\R^r\to\so(\v)=\so(m)$ be of Heisenberg type. Write $\g:=\g(j)$ and view $R$ as an endomorphism
of $\g\wedge\g$ by requiring
$\<R(A,B)C,D\>=\<R(A\wedge B),C\wedge D\>$
for all $A,B,C,D\in\g$,
where the inner product on~$\g\wedge\g$ is defined in the usual way by bilinear extension
of $\<E\wedge F,C\wedge D\>=\<E,C\>\<F,D\>-\<E,D\>\<F,C\>$. Write $R^{\v\wedge\v}:=\Pr_{\v\wedge\v}
\circ R\restr{\v\wedge\v}$, where $\Pr_{\v\wedge\v}:\g\wedge\g\to\v\wedge\v$
denotes orthogonal projection. Then for all $q\in\N$ we have
$$\Tr((R^{\v\wedge\v})^q) = (-\tfrac14)^q\bigl(\tfrac12I_{k_1\ldots k_q|k_1\ldots k_q}-
  \tfrac12I_{k_1\ldots k_qk_1\ldots k_q}+r(2-r+m)^q-r(2-r)^q\bigr).
$$
\end{lemma}

\begin{proof}
As always, let $\{X_1,\ldots,X_m\}$ and $\{Z_1,\ldots,Z_r\}$ be orthonormal bases of $\v$, resp.~$\z$.
For $Z\in\z$, we let $E_Z:=\sum_{k=1}^m X_k\wedge j_ZX_k\in\v\wedge\v$. Note that $E_Z$
is defined independently of the choice of orthonormal basis in~$\v$.
If $Z\in\z$ is a unit vector then
\begin{equation*}
  \begin{split}
  |E_Z|^2&=\tsum_{k,\ell=1}^m\<X_k\wedge j_ZX_k,X_\ell\wedge j_ZX_\ell\>\\
  &=\tsum_{k,\ell=1}^m(\<X_k,X_\ell\>\<j_ZX_k,j_ZX_\ell\>-\<X_k,j_ZX_\ell\>\<X_\ell,j_ZX_k\>)
  =-2\Tr(j_Z^2)=2m.
  \end{split}
\end{equation*}
Using polarization we see that $\{E_{Z_1},\ldots,E_{Z_r}\}\subset\v\wedge\v$
is an orthogonal set of vectors of norm~$\sqrt{2m}$.
Define $\Phi:=\v\wedge\v\to\v\wedge\v$ by
$\Phi(X\wedge Y):=\sum_{\alpha=1}^r j_{Z_\alpha}X\wedge j_{Z_\alpha}Y$.
We obtain, using that $\{j_{Z_\alpha}X_1,\ldots,j_{Z_\alpha}X_m\}$ is again
an orthogonal basis of~$\v$:
\begin{equation*}
  \begin{split}
  \Phi(E_Z)&=\tsum_{\alpha=1}^r\tsum_{k=1}^m(j_{Z_\alpha}X_k\wedge j_{Z_\alpha}j_Z X_k)\\
    &=\tsum_{\alpha=1}^r\tsum_{k=1}^m(-j_{Z_\alpha}
      X_k\wedge j_Z j_{Z_\alpha}X_k-j_{Z_\alpha}X_k\wedge 2\<Z,Z_\alpha\>X_k)=(-r+2)E_Z
  \end{split}
\end{equation*}
for $Z\in\z$. Let $\Pr_\E:\v\wedge\v\to\v\wedge\v$ denote
orthogonal projection to $\E:=\spann\{E_{Z_1},\ldots,E_{Z_r}\}$. Note that $\Phi$
is symmetric. Thus, the previous formula implies $\Phi(\E)=\E$, $\Phi(\E^\perp)=\E^\perp$, and
\begin{equation}
\label{eq:commproj}
\Phi\circ\Pr_\E=\Pr_\E\circ\Phi=(2-r)\Pr_\E.
\end{equation}
On the other hand, for $X,U,Y,V\in\v$, the formula for $\<R(X,U)Y,V\>$
from Lemma~\ref{lem:firstformulas}(ii) easily translates into
\begin{equation*}
  \begin{split}
  \<R(X\wedge U),Y\wedge V\>&=\tsum_{\alpha=1}^r
  \bigl(-\tfrac14\<j_{Z_\alpha}X\wedge j_{Z_\alpha}U,Y\wedge V\>
    -\tfrac18\<X\wedge U,E_{Z_\alpha}\>\<E_{Z_\alpha},Y\wedge V\>\bigr)\\
  &=-\tfrac14\<\Phi(X\wedge U),Y\wedge V\>-\tfrac18\cdot 2m\<\Pr_\E(X\wedge U),Y\wedge V\>
  \end{split}
\end{equation*}
(recall that $|E_{Z_\alpha}|=\sqrt{2m}$), hence
\begin{equation*}
R^{\v\wedge\v}=-\tfrac14(\Phi+m\Pr_\E).
\end{equation*}
Using~\eqref{eq:commproj} and $\Tr(\Pr_\E)=r$ we conclude
$$\Tr((R^{\v\wedge\v})^q)=(-\tfrac14)^q\bigl(\Tr(\Phi^q)+\tsum_{p=1}^q\binom qp(2-r)^{q-p} m^pr\bigr)
  =(-\tfrac14)^q\bigl(\Tr(\Phi^q)+r((2-r+m)^q-(2-r)^q)\bigr).
$$
The statement thus follows from
\begin{equation*}
  \begin{split}
  \Tr(\Phi^q)&=\tsum_{k<\ell}\<\Phi^q(X_k\wedge X_\ell),X_k\wedge X_\ell\>
   =\tfrac12\tsum_{k,\ell=1}^m\<\Phi^q(X_k\wedge X_\ell),X_k\wedge X_\ell\>\\
   &=\tfrac12\tsum_{k,\ell=1}^m\tsum_{\alpha_1,\ldots,\alpha_q=1}^r
     \<j_{Z_{\alpha_1}}\ldots j_{Z_{\alpha_q}}X_k\wedge j_{Z_{\alpha_1}}\ldots j_{Z_{\alpha_q}}X_\ell,X_k\wedge X_\ell\>\\
   &=\tfrac12\tsum_{\alpha_1,\ldots,\alpha_q=1}^r\bigl((\Tr(j_{Z_{\alpha_1}}\ldots j_{Z_{\alpha_q}}))^2-
     \Tr(j_{Z_{\alpha_1}}\ldots j_{Z_{\alpha_q}}j_{Z_{\alpha_1}}\ldots j_{Z_{\alpha_q}})\bigr)\\
   &=\tfrac12(I_{k_1\ldots k_q|k_1\ldots k_q}-I_{k_1\ldots k_qk_1\ldots k_q}).
\end{split}
\end{equation*}
\end{proof}

\begin{proof}[Proof of Proposition~\ref{prop:notcurveq}]
Let $N$ and $N'$ be curvature equivalent; we are going to show that they are locally isometric.
Let $F:\g(j)=\v\oplus\z\to \g(j')=\v'\oplus\z'$
be a euclidean isometry of the associated Lie algebras which intertwines the curvature tensors.
Then $F$ also intertwines the Ricci tensors. Note that here in the Heisenberg type case we have
$\Ric(g(j))\restr{\v}=-\frac r2\Id_\v$ and $\Ric(g(j))\restr{\z}=\frac m4\Id_\z$
by Lemma~\ref{lem:firstformulas}(iii), and similarly
for~$j'$. Since $F$ has to
preserve the eigenspace associated to the negative, resp.~positive eigenvalue, we have $F(\v)=\v'$ and $F(\z)=\z'$;
in particular, $m=m'$ and $r=r'$. The restriction of~$F$ to~$\v$ now induces a linear map from
$\v\wedge\v$ to $\v'\wedge\v'$ which intertwines $R(g(j))^{\v\wedge\v}$ and $R(g(j'))^{\v'\wedge\v'}$;
in particular, these operators have the same trace, and so do their $q$-th powers for any~$q$.
Applying Lemma~\ref{lem:TrRvq} in the special case $q:=r$, we conclude
$$I_{k_1\ldots k_r|k_1\ldots k_r}(j)-I_{k_1\ldots k_rk_1\ldots k_r}(j)=
I_{k_1\ldots k_r|k_1\ldots k_r}(j')-I_{k_1\ldots k_rk_1\ldots k_r}(j').
$$
The second terms on each side of this equation coincide by Proposition~\ref{prop:heisImore}(ii).
Thus, the first terms have to coincide, too. By Proposition~\ref{prop:heisImore}(iii) this implies
that $j$ and~$j'$ are equivalent in the sense of Remark~\ref{rem:isometric}(iii); so $N$ and~$N'$
are indeed locally isometric.
\end{proof}

Together with Remark~\ref{rem:heisexist}(i) and Theorem~\ref{thm:zgreater}(i), the previous
proposition implies:

\begin{theorem}
\label{thm:notcurveqinspiteofinvars}
For any $k\in\N$, there exist pairs of locally homogeneous Riemannian manifolds
which are not curvature equivalent, but do not differ in any curvature invariant
of order up to~$2k$.
\end{theorem}

\section*{Appendix}

\begin{proof}[Proof of Remark~\ref{lem:lowerorder}(iv).]\

\noindent
For $A\in\g$, write $R_A:\g\times\g\ni(B,C)\mapsto R(A,B)C\in\g$, and consider the canonical
extension of $\scp$ to tensors of this form. We start by computing individual formulas
for $\<R_A,R_B\>$ because we will need them below in the proof of Lemma~\ref{lem:nablaric}(ii).
For $U,Y\in\v$ we have, by Lemma~\ref{lem:firstformulas}(ii),
\begin{equation*}
  \begin{split}
  \bullet\ &\<R_U\restr{\v\times\v}\,,R_Y\restr{\v\times\v}\>
    =\tsum_{k,\ell,a=1}^m\<R(U,X_k)X_\ell,X_a\>\<R(Y,X_k)X_\ell,X_a\>\\
  &=\tfrac1{16}\tsum_{k,\ell,a=1}^m\tsum_{\beta,\gamma=1}^r
    \bigl(\<j_{Z_\beta}U,X_a\>\<j_{Z_\beta}X_k,X_\ell\>
    -\<j_{Z_\beta}U,X_\ell\>\<j_{Z_\beta}X_k,X_a\>\\
  &\quad\quad\quad-2\<j_{Z_\beta}U,X_k\>\<j_{Z_\beta}X_\ell,X_a\>\bigr)\cdot\\
  &\quad\quad\quad\cdot\bigl(
    \<j_{Z_\gamma}Y,X_a\>\<j_{Z_\gamma}X_k,X_\ell\>
    -\<j_{Z_\gamma}Y,X_\ell\>\<j_{Z_\gamma}X_k,X_a\>
    -2\<j_{Z_\gamma}Y,X_k\>\<j_{Z_\gamma}X_\ell,X_a\>\bigr)\\
  &=\tfrac1{16}\tsum_{\beta,\gamma=1}^r\bigl(
    (1^2+1^2+2^2)\<j_{Z_\beta}U,j_{Z_\gamma}Y\>\<j_{Z_\beta},j_{Z_\gamma}\>\\
  &\quad\quad\quad
    +(1-2+1-2-2-2)\<j_{Z_\beta}U,j_{Z_\gamma}j_{Z_\beta}j_{\Z_\gamma}Y\>\bigr)\\
  &=\tfrac38\tsum_{\beta,\gamma=1}^r\bigl(\<j_{Z_\beta}j_{Z_\gamma}U,Y\>\Tr(
    j_{Z_\beta}j_{Z_\gamma})+\<j_{Z_\beta}j_{Z_\gamma}j_{Z_\beta}j_{Z_\gamma}U,Y\>
    \bigr),\\
  \bullet\ &\<R_U\restr{\v\times\z}\,,R_Y\restr{\v\times\z}\>
    =\tsum_{k=1}^m\tsum_{\beta,\gamma=1}^r
    \<R(U,X_k)Z_\beta,Z_\gamma\>\<R(Y,X_k)Z_\beta,Z_\gamma\>\\
  &=\tfrac1{16}\tsum_{k=1}^m\tsum_{\beta,\gamma=1}^r
    \<[j_{Z_\beta},j_{Z_\gamma}]U,X_k\>\<[j_{Z_\beta},j_{Z_\gamma}]Y,X_k\>
  =-\tfrac1{16}\tsum_{\beta,\gamma=1}^r \<[j_{Z_\beta},j_{Z_\gamma}]^2 U,Y\>\\
  &=-\tfrac18\tsum_{\beta,\gamma=1}^r\<j_{Z_\beta}j_{Z_\gamma}j_{Z_\beta}j_{Z_\gamma}U,Y\>
  +\tfrac18\tsum_{\beta=1}^r \<j_{Z_\beta}Jj_{Z_\beta}U,Y\>,\\
  \bullet\ &\<R_U\restr{\z\times\v}\,,R_Y\restr{\z\times\v}\>
  +\<R_U\restr{\z\times\z}\,,R_Y\restr{\z\times\z}\>
    =2\tsum_{k=1}^m\tsum_{\beta,\gamma=1}^r\<R(U,Z_\beta)Z_\gamma,X_k\>\<R(Y,Z_\beta)Z_\gamma,X_k\>\quad\quad\;\\
  &=\tfrac18\tsum_{k=1}^m\tsum_{\beta,\gamma=1}^r\<j_{Z_\gamma}U,j_{Z_\beta}X_k\>\<j_{Z_\gamma}Y,j_{Z_\beta}X_k\>
    =\tfrac18\tsum_{\beta,\gamma=1}^r\<j_{Z_\beta}j_{Z_\gamma}U,j_{Z_\beta}j_{Z_\gamma}Y\>\\
  &=\tfrac18\tsum_{\beta=1}^r\<j_{Z_\beta}Jj_{Z_\beta}U,Y\>.
  \end{split}
\end{equation*}
Hence,
\begin{equation}
\label{eq:RURY}
\<R_U,R_Y\>=
\tsum_{\beta,\gamma=1}^r\bigl(\tfrac38\<j_{Z_\beta}j_{Z_\gamma}U,Y\>\Tr(j_{Z_\beta}j_{Z_\gamma})
  +\tfrac14\<j_{Z_\beta}j_{Z_\gamma}j_{Z_\beta}j_{Z_\gamma}U,Y\>\bigr)
  +\tfrac14\tsum_{\beta=1}^r\<j_{Z_\beta}Jj_{Z_\beta}U,Y\>.
\end{equation}
For $W\in\z$ we have $R_W\restr{\z\times\z}=0$ and
\begin{equation*}
  \begin{split}
  \bullet\ &|R_W\restr{\v\times\v}|^2+|R_W\restr{\v\times\z}|^2
    =2\tsum_{k,\ell=1}^m\tsum_{\alpha=1}^r\<R(W,X_k)X_\ell,Z_\alpha\>^2
    =\tfrac18\tsum_{k,\ell=1}^m\tsum_{\alpha=1}^r\<j_W X_\ell,j_{Z_\alpha}X_k\>^2\\
  &=\tfrac18\tsum_{\alpha=1}^2|j_Wj_{Z_\alpha}|^2=\tfrac18\Tr(Jj_W^2),\\
  \bullet\ &|R_W\restr{\z\times\v}|^2=\tsum_{k,\ell=1}^m\tsum_{\alpha=1}^r\<R(W,Z_\alpha)X_k,X_\ell\>^2
    =\tfrac1{16}\tsum_{k,\ell=1}^m\tsum_{\alpha=1}^r\<[j_W,j_{Z_\alpha}]X_k,X_\ell\>^2\\
  &=\tfrac1{16}|j_Wj_{Z_\alpha}-j_{Z_\alpha}j_W|^2=\tfrac18\Tr(Jj_W^2)-\tfrac18\tsum_{\alpha=1}^r
  \Tr(j_Wj_{Z_\alpha}j_Wj_{Z_\alpha})
  \end{split}
\end{equation*}
and thus for $Z,W\in\z$, using polarization,
\begin{equation}
\label{eq:RZRW}
\<R_Z,R_W\>=\tfrac18\Tr(J(j_Zj_W+j_Wj_Z))-\tfrac18\tsum_{\alpha=1}^r\Tr(j_Zj_{Z_\alpha}j_Wj_{Z_\alpha}).
\end{equation}

Moreover, $\<R_X,R_Z\>=0$ for all $X\in\v$, $Z\in\z$ by Lemma~\ref{lem:firstformulas}(ii).
Using \eqref{eq:RURY} and~\eqref{eq:RZRW}, we obtain
\begin{equation*}
  \begin{split}
  |R|^2&=\tsum_{k=1}^m\<R_{X_k},R_{X_k}\>+\tsum_{\alpha=1}^r\<R_{Z_\alpha},R_{Z_\alpha}\>
    =\tfrac38I_{\alpha\beta|\alpha\beta}+\tfrac14I_{\alpha\beta\alpha\beta}+\tfrac14I_{\alpha\alpha\beta\beta}
    +\tfrac14I_{\alpha\alpha\beta\beta}-\tfrac18I_{\alpha\beta\alpha\beta},
  \end{split}
\end{equation*}
from which the statement follows.
\end{proof}

\begin{proof}[Proof of Lemma~\ref{lem:nablaric}.]\

\noindent
(i) First note that by Lemma~\ref{lem:firstformulas},
\begin{equation*}
  \begin{split}
  &\tsum_{k,\ell=1}^m\tsum_{\alpha,\beta=1}^r\ric(X_k,X_\ell)\ric(Z_\alpha,Z_\beta)
  \<R(X_k,X_\ell)Z_\alpha,Z_\beta\>\\
  &=\tfrac1{32}\tsum_{k,\ell=1}^m\tsum_{\alpha,\beta=1}^r
    \<JX_k,X_\ell\>\Tr(j_{Z_\alpha}j_{Z_\beta})\<[j_{Z_\alpha},j_{Z_\beta}]X_k,X_\ell\>
  =\frac1{32}\Tr(j_{Z_\alpha}j_{Z_\beta})\Tr([j_{Z_\alpha},j_{Z_\beta}]J)=0,
  \end{split}
\end{equation*}
and $\ric(\v,\z)=0$. Thus, recalling that $J$ is symmetric,
\begin{equation*}
  \begin{split}
  (*)={}&\tsum_{k,\ell,a,b=1}^m\ric(X_k,X_\ell)\ric(X_a,X_b)\<R(X_k,X_a)X_\ell,X_b\>\\
  ={}&\tfrac14\tsum_{k,\ell,a,b=1}^m\<JX_k,X_\ell\>\<JX_a,X_b\>\cdot\\
  &\cdot\tsum_{\beta=1}^r
    \bigl(\tfrac14\<j_{Z_\beta}X_a,X_\ell\>\<j_{Z_\beta}X_k,X_b\>-\tfrac14
    \<j_{Z_\beta}X_k,X_\ell\>\<j_{Z_\beta}X_a,X_b\>-\tfrac12\<j_{Z_\beta}X_k,X_a\>\<j_{Z_\beta}X_\ell,X_b\>\bigr)\\
  ={}&\tfrac1{16}\tsum_{k,a=1}^m\sum_{\beta=1}^r\bigl(\<j_{Z_\beta}X_a,JX_k\>\<j_{Z_\beta}X_k,JX_a\>
   -\<j_{Z_\beta}X_k,JX_k\>\<j_{Z_\beta}X_a,JX_a\>\\
   &\hphantom{\tsum_{k,a=1}^m\sum_{\beta=1}^r\bigl(}-2\<j_{Z_\beta}X_k,X_a\>\<j_{Z_\beta}JX_k,JX_a\>\bigr)\\
  ={}&\tfrac1{16}\tsum_{\beta=1}^r(\<-j_{Z_\beta}J,Jj_{Z_\beta}\>-0-2\<j_{Z_\beta},Jj_{Z_\beta}J\>)
  =\tfrac1{16}\tsum_{\beta=1}^r3\Tr(Jj_{Z_\beta}Jj_{Z_\beta})
  =\tfrac3{16}I_{\alpha\alpha\beta\gamma\gamma\beta}.
  \end{split}
\end{equation*}

\noindent
(ii)
By definition of $(**)$ and by Lemma~\ref{lem:firstformulas}(iii),
\begin{equation*}
  (**)=\tsum_{k,\ell=1}^m\tfrac12\<JX_k,X_\ell\>\<R_{X_k},R_{X_\ell}\>
  -\tsum_{\alpha,\beta=1}^r\tfrac14\Tr(j_{Z_\alpha}j_{Z_\beta})\<R_{Z_\alpha},R_{Z_\beta}\>.
\end{equation*}
Thus, using \eqref{eq:RURY} and~\eqref{eq:RZRW},
\begin{equation*}
  \begin{split}
  (**)={}&\tfrac12\tsum_{k=1}^m\Bigl(\tsum_{\beta,\gamma=1}^r
   \bigl(\tfrac38\<j_{Z_\beta}j_{Z_\gamma}X_k,JX_k\>\Tr(j_{Z_\beta}j_{Z_\gamma})
   +\tfrac14\<j_{Z_\beta}j_{Z_\gamma}j_{Z_\beta}j_{Z_\gamma}X_k,JX_k\>\bigr)\\
   &\hphantom{\tsum_{k=1}^m\Bigl(}
    +\tfrac14\tsum_{\beta=1}^r\<j_{Z_\beta}Jj_{Z_\beta}X_k,JX_k\>\Bigr)\\
   &-\tfrac14\tsum_{\beta,\gamma=1}^r\Tr(j_{Z_\beta}j_{Z_\gamma})\bigl(\tfrac18\Tr(J(j_{Z_\beta}j_{Z_\gamma}
    +j_{Z_\gamma}j_{Z_\beta}))-\tfrac18\tsum_{\alpha=1}^r\Tr(j_{Z_\beta}j_{Z_\alpha}j_{Z_\gamma}j_{Z_\alpha})\bigr)\\
   ={}&\tfrac3{16}I_{\alpha\alpha\beta\gamma|\beta\gamma}+\tfrac18I_{\alpha\alpha\beta\gamma\beta\gamma}
    +\tfrac18I_{\alpha\alpha\beta\gamma\gamma\beta}
    -\tfrac1{16}I_{\alpha\alpha\beta\gamma|\beta\gamma}+\tfrac1{32}I_{\beta\alpha\gamma\alpha|\beta\gamma}\\
    ={}&\tfrac18I_{\alpha\alpha\beta\gamma|\beta\gamma}+\tfrac18I_{\alpha\alpha\beta\gamma\beta\gamma}
    +\tfrac18I_{\alpha\alpha\beta\gamma\gamma\beta}+\tfrac1{32}I_{\alpha\gamma\beta\gamma|\alpha\beta}.
  \end{split}
\end{equation*}

\noindent
(iii)
For $X,Y\in\v$ and $Z\in\z$, we have
$(\nabla_Y\ric)\restr{\v\times\v}=0$, $(\nabla_Y\ric)\restr{\z\times\z}=0$ and
\begin{equation*}
  \begin{split}
  ((\nabla_Y\ric)(Z,X))^2
  &=(-\tfrac12\ric(X,j_ZY)+\tfrac12\ric([Y,X],Z))^2\\
  &=(-\tfrac14\<JX,j_ZY\>-\tfrac18\tsum_{\beta=1}^r\Tr(j_{Z_\beta}j_Z)\<j_{Z_\beta}Y,X\>)^2,\text{ hence}
  \end{split}
\end{equation*}
\begin{equation*}
  \begin{split}
  |(\nabla_Y\ric)|^2={}&2\tsum_{\ell=1}^m\tsum_{\gamma=1}^r\bigl(\tfrac1{16}\<JX_\ell,j_{Z_\gamma}Y\>^2
  +\tfrac1{16}
  \tsum_{\beta=1}^r\<JX_\ell,j_{Z_\gamma}Y\>\Tr(j_{Z_\beta}j_{Z_\gamma})\<j_{Z_\beta}Y,X_\ell\>\\
  &\hphantom{2\tsum_{\ell=1}^m\tsum_{\gamma=1}^r}
  +\tfrac1{64}\tsum_{\alpha,\beta=1}^r\Tr(j_{Z_\alpha}j_{Z_\gamma})\Tr(j_{Z_\beta}j_{Z_\gamma})
  \<j_{Z_\alpha}Y,X_\ell\>\<j_{Z_\beta}Y,X_\ell\>\bigr)\\
  ={}&\tsum_{\gamma=1}^r\bigl(\tfrac18|Jj_{Z_\gamma}Y|^2
  +\tfrac18\tsum_{\beta=1}^r\Tr(j_{Z_\beta}j_{Z_\gamma})\<j_{Z_\beta}Y,Jj_{Z_\gamma}Y\>\\
  &\hphantom{\tsum_{\gamma=1}^r\bigl(}
  +\tfrac1{32}\tsum_{\alpha,\beta=1}^r\Tr(j_{Z_\alpha}j_{Z_\gamma})\Tr(j_{Z_\beta}j_{Z_\gamma})
  \<j_{Z_\alpha}Y,j_{Z_\beta}Y\>\bigr).\text{ Thus,}\\
  |(\nabla\ric)\restr{\v}|^2={}&\tsum_{\gamma=1}^r\bigl(\tfrac18|Jj_{Z_\gamma}|^2
  -\tfrac18\tsum_{\beta=1}^r\Tr(j_{Z_\beta}j_{Z_\gamma})\Tr(Jj_{Z_\gamma}j_{Z_\beta})\\
  &\hphantom{\tsum_{\gamma=1}^r\bigl(}-\tfrac1{32}
    \tsum_{\alpha,\beta=1}^r\Tr(j_{Z_\alpha}j_{Z_\gamma})\Tr(j_{Z_\beta}j_{Z_\gamma})
    \Tr(j_{Z_\alpha}j_{Z_\beta})\bigr)\\
  ={}&-\tfrac18\Tr(J^3)-\tfrac18I_{\alpha\alpha\beta\gamma|\beta\gamma}
    -\tfrac1{32}I_{\alpha\beta|\alpha\gamma|\beta\gamma},
  \end{split}
\end{equation*}
where $|(\nabla\ric)\restr{\v}|^2$ denotes $\sum_{k=1}^m|\nabla_{X_k}\ric|^2$.
For $X,Y\in\v$ and $W\in\z$, we have
$(\nabla_W\ric)\restr{\v\times\z}=0$, $(\nabla_W\ric)\restr{\z\times\v}=0$, $(\nabla_W\ric)\restr{\z\times\z}=0$, and
\begin{equation*}
  \begin{split}
  ((\nabla_W\ric)(X,Y))^2
  &=(-\tfrac12\ric(j_WX,Y)-\tfrac12\ric(X,j_WY))^2
  =(\tfrac14\<Jj_WX,Y\>+\tfrac14\<JX,j_WY\>)^2,\text{ hence}
  \end{split}
\end{equation*}
\begin{equation*}
  \begin{split}
  |\nabla_W\ric|^2={}&\tfrac1{16}\tsum_{k,\ell=1}^m
    \bigl(\<Jj_WX_k,X_\ell\>^2+\<JX_k,j_WX_\ell\>^2+2\<Jj_WX_k,X_\ell\>\<JX_k,j_WX_\ell\>\bigr)\\
  ={}&\tfrac18|Jj_W|^2-\tfrac18\<Jj_W,j_WJ\>.\text{ Thus,}\\
  |(\nabla\ric)\restr{\z}|^2={}&\tsum_{\beta=1}^r
    \bigl(\tfrac18|Jj_{Z_\beta}|^2+\tfrac18\Tr(Jj_{Z_\beta}Jj_{Z_\beta})\bigr)
  =-\tfrac18\Tr(J^3)+\tfrac18I_{\alpha\alpha\beta\gamma\gamma\beta},
  \end{split}
\end{equation*}
where $|(\nabla\ric)\restr{\z}|^2$ denotes $\sum_{\alpha=1}^r|\nabla_{Z_\alpha}\ric|^2$.
So,
$$|\nabla\ric|^2=|(\nabla\ric)\restr{\v}|^2+|(\nabla\ric)\restr{\z}|^2
=-\tfrac14\Tr(J^3)-\tfrac18I_{\alpha\alpha\beta\gamma|\beta\gamma}
-\tfrac1{32}I_{\alpha\beta|\alpha\gamma|\beta\gamma}+\tfrac18I_{\alpha\alpha\beta\gamma\gamma\beta}.
$$
\end{proof}

\begin{proof}[Proof of Lemma~\ref{lem:nablar}.]\

\noindent
(i)
The various contributions
\begin{equation}
\label{eq:contribnablar}
\tsum\<(\nabla_AR)(B,C)D,E\>^2
\end{equation}
to $|\nabla R|^2$,
where each of $A,B,C,D,E$ runs through either the orthonormal basis $\{X_1,\ldots,X_m\}$ of~$\v$ or
the orthonormal basis $\{Z_1,\ldots,Z_r\}$ of~$\z$, can by Remark~\ref{rem:formj}
be nonzero only in the cases where~$\v$ occurs an even number of times.
If $\v$ occurs exactly twice then, again by Remark~\ref{rem:formj}, each
$\<(\nabla_AR)(B,C)D,E\>$ is a linear combination of terms of the type
$\<j_{Z_{\alpha_1}}j_{Z_{\alpha_2}}j_{Z_{\alpha_3}}X_{\ell_1}X_{\ell_2}\>$,
where $(X_{\ell_1},X_{\ell_2},Z_{\alpha_1},Z_{\alpha_2},Z_{\alpha_3})$ is just some permutation
of $(A,B,C,D,E)$. The sum in~\eqref{eq:contribnablar} will thus be a linear
combination of sums of the type
$$\tsum\<j_{Z_{\alpha_{s_1}}}j_{Z_{\alpha_{s_2}}}j_{Z_{\alpha_{s_3}}}X_{\ell_{u_1}},X_{\ell_{u_2}}\>
\<j_{Z_{\alpha_{s_4}}}j_{Z_{\alpha_{s_5}}}j_{Z_{\alpha_{s_6}}}X_{\ell_{u_1}},X_{\ell_{u_2}}\>,
$$
where $(s_1,s_2,s_3)$ and $(s_4,s_5,s_6)$ are the same up to permuation,
and the summation is done w.r.t.~pairs of equal indices $s_i$ and~$u_j$.
Summation over equal pairs of $u_j$ yields
$$-\tsum\Tr(j_{Z_{\alpha_{s_6}}}j_{Z_{\alpha_{s_5}}}j_{Z_{\alpha_4}}j_{Z_{\alpha_{s_1}}}
  j_{Z_{\alpha_{s_2}}}j_{Z_{\alpha_{s_3}}}),
$$
which equals $-I_{s_6s_5s_4s_1s_2s_3}$, one of our invariants from Definition~\ref{def:Ik}
in which only subtuples of even length occur (in this case, only
one subtuple, and this one of length six). Hence, sums as in~\eqref{eq:contribnablar} with exactly two
occurrences of~$\v$ contribute only to the term~$L_1$ from the assertion.
Therefore it remains to consider sums as in~\eqref{eq:contribnablar} with
exactly four occurrences of~$\v$.

Due to the symmetries of~$R$, the contribution of such sums is equal to
\begin{equation}
\label{eq:twopartsnablar}
\tsum\<(\nabla_WR)(X,Y)U,V\>^2
+4\tsum\<(\nabla_XR)(W,Y)U,V\>^2,
\end{equation}
where both sums are taken over $X,Y,U,V\in\{X_1,\ldots,X_m\}$, $W\in\{Z_1,\ldots,Z_r\}$.
For the first term in~\eqref{eq:twopartsnablar}, we note using Lemma~\ref{lem:firstformulas}(i),~(ii)
and the skew-symmetry of the maps $j_W$,~$j_{Z_\alpha}$
that $\<(\nabla_WR)(X,Y)U,V\>$ is the sum of the following twelve summands:
\begin{itemize}
\item[(a)]
$-\<\nabla_W\nabla_X\nabla_YU,V\>=-\frac18\tsum_{\beta}\<j_Wj_{Z_\beta}X,V\>\<j_{Z_\beta}Y,U\>$,
\item[(b)]
$\hphantom{-\<\nabla_W\nabla_X\nabla_YU,V\>}\llap{$\<\nabla_W\nabla_Y\nabla_XU,V\>$}
  =\hphantom{-}\frac18\tsum_\beta\<j_Wj_{Z_\beta}Y,V\>\<j_{Z_\beta}X,U\>$,
\item[(c)]
$\hphantom{-\<\nabla_W\nabla_X\nabla_YU,V\>}\llap{$\<\nabla_W\nabla_{[X,Y]}U,V\>$}
  =\hphantom{-}\frac14\tsum_\beta\<j_Wj_{Z_\beta}U,V\>\<j_{Z_\beta}X,Y\>$,
\item[(d)]
$\hphantom{-\<\nabla_W\nabla_X\nabla_YU,V\>}\llap{$\<\nabla_{\nabla_WX}\nabla_YU,V\>$}
  =\hphantom{-}\frac18\tsum_\beta\<j_{Z_\beta}j_WX,V\>\<j_{Z_\beta}Y,U\>$,
\item[(e)]
$\hphantom{-\<\nabla_W\nabla_X\nabla_YU,V\>}\llap{$-\<\nabla_Y\nabla_{\nabla_WX}U,V\>$}
  =-\frac18\tsum_\beta\<j_{Z_\beta}j_WX,U\>\<j_{Z_\beta}Y,V\>$,
\item[(f)]
$\hphantom{-\<\nabla_W\nabla_X\nabla_YU,V\>}\llap{$-\<\nabla_{[\nabla_WX,Y]}U,V\>$}
  =-\frac14\tsum_\beta\<j_{Z_\beta}j_WX,Y\>\<j_{Z_\beta}U,V\>$,
\item[(g)]
$\hphantom{-\<\nabla_W\nabla_X\nabla_YU,V\>}\llap{$\<\nabla_X\nabla_{\nabla_WY}U,V\>$}
  =\hphantom{-}\frac18\tsum_\beta\<j_{Z_\beta}j_WY,U\>\<j_{Z_\beta}X,V\>$,
\item[(h)]
$\hphantom{-\<\nabla_W\nabla_X\nabla_YU,V\>}\llap{$-\<\nabla_{\nabla_WY}\nabla_XU,V\>$}
  =-\frac18\tsum_\beta\<j_{Z_\beta}j_WY,V\>\<j_{Z_\beta}X,U\>$,
\item[(i)]
$\hphantom{-\<\nabla_W\nabla_X\nabla_YU,V\>}\llap{$-\<\nabla_{[X,\nabla_WY]}U,V\>$}
  =\hphantom{-}\frac14\tsum_\beta\<j_Wj_{Z_\beta}X,Y\>\<j_{Z_\beta}U,V\>$,
\item[(j)]
$\hphantom{-\<\nabla_W\nabla_X\nabla_YU,V\>}\llap{$\<\nabla_X\nabla_Y\nabla_WU,V\>$}
  =-\frac18\tsum_\beta\<j_Wj_{Z_\beta}Y,U\>\<j_{Z_\beta}X,V\>$,
\item[(k)]
$\hphantom{-\<\nabla_W\nabla_X\nabla_YU,V\>}\llap{$-\<\nabla_Y\nabla_X\nabla_WU,V\>$}
  =\hphantom{-}\frac18\tsum_\beta\<j_Wj_{Z_\beta}X,U\>\<j_{Z_\beta}Y,V\>$,
\item[(l)]
$\hphantom{-\<\nabla_W\nabla_X\nabla_YU,V\>}\llap{$-\<\nabla_{[X,Y]}\nabla_WU,V\>$}
  =-\frac14\tsum_\beta\<j_{Z_\beta}j_WU,V\>\<j_{Z_\beta}X,Y\>$,
\end{itemize}
where the sums are taken over $\beta\in\{1,\ldots,r\}$.
Now $\<(\nabla_WR)(X,Y)U,V\>$ is the square of the sum of the twelve terms.
The square of each single one of them will just lead to a contribution to~$L_1$:
For example, the square of the term in~(a) is $\frac1{64}$ times
$$\tsum_{\beta,\gamma}
\<j_Wj_{Z_\beta}X,V\>\<j_{Z_\beta}Y,U\>\<j_Wj_{Z_\gamma}X,V\>\<j_{Z_\gamma}Y,U\>,
$$
which after summation over $X,Y,U,V\in\{X_1,\ldots,X_m\}$ gives
$-\sum_{\beta,\gamma}\Tr(j_Wj_Wj_{Z_\beta}j_{Z_\gamma})\Tr(j_{Z_\beta}j_{Z_\gamma})$;
summation over $W\in\{Z_1,\ldots,Z_r\}$ thus yields
$-I_{\alpha\alpha\beta\gamma|\beta\gamma}$, another invariant
in which only subtuples of even lengths (here, four and two) occur.

Next, consider the product of the terms in (a) and~(b) which is $-\frac1{64}$ times
$$\tsum_{\beta,\gamma}
\<j_Wj_{Z_\beta}X,V\>\<j_{Z_\beta}Y,U\>\<j_Wj_{Z_\gamma}Y,V\>\<j_{Z_\gamma}X,U\>.
$$
One easily checks that this leads to an invariant with just one subtuple of length six.
The technical reason is that here, there is no way to group the four factors into subsets which
would not be linked to each other by the occurrence of any common vectors from $\{X,Y,U,V\}$.

The only pairings of different terms from (a)--(l) above where this does not happen are the
following twelve:
\begin{equation}
\label{eq:pairings}
((a)\text{ or }(d))\longleftrightarrow((g)\text{ or }(j)),\quad
((b)\text{ or }(h))\longleftrightarrow((e)\text{ or }(k)),\quad
((c)\text{ or }(l))\longleftrightarrow((f)\text{ or }(i)),
\end{equation}
For example, the product of the terms in (a) and (g) is
\begin{equation}
\label{eq:a-g}
-\tfrac1{64}\tsum_{\beta,\gamma}
\<j_Wj_{Z_\beta}X,V\>\<j_{Z_\beta}Y,U\>\<j_{Z_\gamma}j_WY,U\>\<j_{Z_\gamma}X,V\>
\end{equation}
which after summation over $X,Y,U,V$ becomes
$-\frac1{64}\tsum_{\beta,\gamma}\Tr(j_{Z_\gamma}j_Wj_{Z_\beta})\Tr(j_{Z_\beta}j_{Z_\gamma}j_W)$.
Summation over $W$ finally yields $-\tfrac1{64}I_{\gamma\alpha\beta|\beta\gamma\alpha}
=-\tfrac1{64}I_{\alpha\beta\gamma|\alpha\beta\gamma}$. Similarly, the product of the
terms in (a) and (j) is
$$\tfrac1{64}\tsum_{\beta,\gamma}
\<j_Wj_{Z_\beta}X,V\>\<j_{Z_\beta}Y,U\>\<j_Wj_{Z_\gamma}Y,U\>\<j_{Z_\gamma}X,V\>.
$$
which gives $-\frac1{64}\Tr(j_{Z_\gamma} j_W j_{Z_\beta})\Tr(j_{Z_\gamma} j_W j_{\beta})
=-\frac1{64}I_{\alpha\beta\gamma|\alpha\beta\gamma}$ again.
For each of the pairings from \eqref{eq:pairings}, note that whenever the two
terms to be paired differ in sign, they also differ in the order of $j_W$ and $j_{Z_\beta}$
in their first factors. Just as we saw for $(a)\leftrightarrow(g)$ and $(a)\leftrightarrow(j)$,
this leads each time to a negative multiple of $I_{\alpha\beta\gamma|\alpha\beta\gamma}$.
Altogether, we obtain
$$2\cdot\bigl(4\cdot(-\tfrac1{64})+4\cdot(-\tfrac1{64})+4\cdot(-\tfrac1{16})\bigr)I_{\alpha\beta\gamma|\alpha\beta\gamma}
=-\tfrac34 I_{\alpha\beta\gamma|\alpha\beta\gamma}
$$
as the contribution to $|\nabla R|^2$ of the first summand in~\eqref{eq:twopartsnablar}, apart from
its contributions to~$L_1$.

For the second summand in~\eqref{eq:twopartsnablar}, we compute that $\<(\nabla_X R)(W,Y)U,V\>$ is
the sum of
\begin{itemize}
\item[(a')]
$-\<\nabla_X\nabla_W\nabla_YU,V\>
  =\hphantom{-}0$,
\item[(b')]
$\hphantom{-\<\nabla_W\nabla_X\nabla_YU,V\>}\llap{$\<\nabla_X\nabla_Y\nabla_WU,V\>$}
  =-\frac18\tsum_\beta\<j_Wj_{Z_\beta}Y,U\>\<j_{Z_\beta}X,V\>$,
\item[(c')]
$\hphantom{-\<\nabla_W\nabla_X\nabla_YU,V\>}\llap{$\<\nabla_X\nabla_{[W,Y]}U,V\>$}
  =\hphantom{-}0$,
\item[(d')]
$\hphantom{-\<\nabla_W\nabla_X\nabla_YU,V\>}\llap{$\<\nabla_{\nabla_XW}\nabla_YU,V\>$}
  =\hphantom{-}\frac18\tsum_\beta\<j_{Z_\beta}j_WX,V\>\<j_{Z_\beta}Y,U\>$,
\item[(e')]
$\hphantom{-\<\nabla_W\nabla_X\nabla_YU,V\>}\llap{$-\<\nabla_Y\nabla_{\nabla_XW}U,V\>$}
  =-\frac18\tsum_\beta\<j_{Z_\beta}j_WX,U\>\<j_{Z_\beta}Y,V\>$,
\item[(f')]
$\hphantom{-\<\nabla_W\nabla_X\nabla_YU,V\>}\llap{$-\<\nabla_{[\nabla_XW,Y]}U,V\>$}
  =-\frac14\tsum_\beta\<j_{Z_\beta}j_WX,Y\>\<j_{Z_\beta}U,V\>$,
\item[(g')]
$\hphantom{-\<\nabla_W\nabla_X\nabla_YU,V\>}\llap{$\<\nabla_W\nabla_{\nabla_XY}U,V\>$}
  =\hphantom{-}\frac18\tsum_\beta\<j_Wj_{Z_\beta}U,V\>\<j_{Z_\beta}X,Y\>$,
\item[(h')]
$\hphantom{-\<\nabla_W\nabla_X\nabla_YU,V\>}\llap{$-\<\nabla_{\nabla_XY}\nabla_WU,V\>$}
  =-\frac18\tsum_\beta\<j_{Z_\beta}j_WU,V\>\<j_{Z_\beta}X,Y\>$,
\item[(i')]
$\hphantom{-\<\nabla_W\nabla_X\nabla_YU,V\>}\llap{$-\<\nabla_{[W,\nabla_XY]}U,V\>$}
  =\hphantom{-}0$,
\item[(j')]
$\hphantom{-\<\nabla_W\nabla_X\nabla_YU,V\>}\llap{$\<\nabla_W\nabla_Y\nabla_XU,V\>$}
  =\hphantom{-}\frac18\tsum_\beta\<j_Wj_{Z_\beta}Y,V\>\<j_{Z_\beta}X,U\>$,
\item[(k')]
$\hphantom{-\<\nabla_W\nabla_X\nabla_YU,V\>}\llap{$-\<\nabla_Y\nabla_W\nabla_XU,V\>$}
  =\hphantom{-}0$,
\item[(l')]
$\hphantom{-\<\nabla_W\nabla_X\nabla_YU,V\>}\llap{$-\<\nabla_{[W,Y]}\nabla_XU,V\>$}
  =\hphantom{-}0$.
\end{itemize}
Similarly as above,
the only pairings which do not just contribute to~$L_1$ now are
$$(b') \longleftrightarrow (d'),\quad
  (e') \longleftrightarrow (j'),\quad
  (f') \longleftrightarrow ((g')\text{ or }(h')).
$$
Again, each of these pairings gives a negative multiple of~$I_{\alpha\beta\gamma|\alpha\beta\gamma}$.
All in all, we obtain
$$4\cdot 2\cdot\bigl(-\tfrac1{64}-\tfrac1{64}+2\cdot(-\tfrac1{32})\bigr)I_{\alpha\beta\gamma|\alpha\beta\gamma}
  =-\tfrac34 I_{\alpha\beta\gamma|\alpha\beta\gamma}
$$
as the contribution to $|\nabla R|^2$ of the second summand in~\eqref{eq:twopartsnablar}, apart from its
contributions to~$L_1$. The statement now follows by $-\frac34-\frac34=-\frac32$.

(ii)
The various contributions
\begin{equation}
\label{eq:contribrhat}
\tsum\<R(A,B)C,D\>\<R(C,D)E,F\>\<R(E,F)A,B\>
\end{equation}
to $\hat R$, where each of $A,B,C,D,E,F$ runs through either the orthonormal
basis $\{X_1,\ldots, X_m\}$ of~$\v$ or the orthonormal basis $\{Z_1,\ldots,Z_r\}$ of~$\z$,
can by Lemma~\ref{lem:firstformulas}(ii) be nonzero only in the cases where each of the
tuples
\begin{equation}
\label{eq:tuplesrhat}
(A,B,C,D),\quad (C,D,E,F), \quad (E,F,A,B)
\end{equation}
contains either two or four vectors from~$\v$.

If each of them contains exactly two vectors from~$\v$, then each summand in~\eqref{eq:contribrhat}
is, again by Lemma~\ref{lem:firstformulas}(ii), a linear combination of products
of three terms of the form $\<j_{Z_{\alpha_1}}j_{Z_{\alpha_2}}X_{\ell_1},X_{\ell_2}\>$.
The sum in~\eqref{eq:contribrhat} will thus be a linear combination of sums of the type
$$\tsum\<j_{Z_{\alpha_{s_1}}}j_{Z_{\alpha_{s_2}}}X_{\ell_{u_1}},X_{\ell_{u_2}}\>
\<j_{Z_{\alpha_{s_3}}}j_{Z_{\alpha_{s_4}}}X_{\ell_{u_3}},X_{\ell_{u_4}}\>
\<j_{Z_{\alpha_{s_5}}}j_{Z_{\alpha_{s_6}}}X_{\ell_{u_5}},X_{\ell_{u_6}}\>,
$$
with each $s_i$ and each $u_j$ occurring exactly twice. Here, summation over equal pairs of~$u_j$
will obviously always lead to invariants $I_{k_1\ldots k_\lambda|\ldots|k_\mu\ldots k_6}$ in which
all subtuples are of even length ($6$, or $4$~and~$2$, or three times~$2$).
Hence, such sums will contribute only to the term~$L_2$ from the assertion.

So let at least one of the tuples from~\eqref{eq:tuplesrhat} consists of vectors in~$\v$.
If $A,B,C,D\in\v$ then either $E,F$ must both be in~$\v$ or both in~$\z$. Therefore, the
contributions of sums as in~\eqref{eq:contribrhat} where at least one of the
tuples from~\eqref{eq:tuplesrhat} consists of vectors in~$\v$ is equal to
\begin{equation}
\label{eq:twopartsrhat}
\tsum\<R(X,Y)U,V\>\<R(U,V)S,T\>\<R(S,T)X,Y\>
+3\tsum\<R(X,Y)U,V\>\<R(U,V)Z,W\>\<R(Z,W)X,Y\>,
\end{equation}
where the first sum is taken over $X,Y,U,V,S,T\in\{X_1,\ldots,X_m\}$ and the second sum over
$X,Y,U,V\allowbreak\in\{X_1,\ldots,X_m\}$ and $Z,W\in\{Z_1,\ldots,Z_r\}$.
For the first term in~\eqref{eq:twopartsrhat}, we note that
\begin{equation}
\label{eq:threetimesthree}
  \begin{split}
  \<R(X,Y)U,&V\>\<R(U,V)S,T\>\<R(S,T)X,Y\>=\\
  \tsum_{\alpha,\beta,\gamma=1}^r\hphantom{\cdot}
    &\bigl(\tfrac14\<j_{Z_\alpha}Y,U\>\<j_{Z_\alpha}X,V\>
    -\tfrac14\<j_{Z_\alpha}X,U\>\<j_{Z_\alpha}Y,V\>
    -\tfrac12\<j_{Z_\alpha}X,Y\>\<j_{Z_\alpha}U,V\>\bigr)\\
  \cdot&\bigl(\tfrac14\<j_{Z_\beta}V,S\>\<j_{Z_\beta}U,T\>
    -\tfrac14\<j_{Z_\beta}U,S\>\<j_{Z_\beta}V,T\>
    -\tfrac12\<j_{Z_\beta}U,V\>\<j_{Z_\beta}S,T\>\bigr)\\
  \cdot&\bigl(\tfrac14\<j_{Z_\gamma}T,X\>\<j_{Z_\gamma}S,Y\>
    -\tfrac14\<j_{Z_\gamma}S,X\>\<j_{Z_\gamma}T,Y\>
    -\tfrac12\<j_{Z_\gamma}S,T\>\<j_{Z_\gamma}X,Y\>\bigr)
\end{split}
\end{equation}
for $X,Y,U,V,S,T\in\v$.
For $a,b,c\in\{1,2,3\}$ denote by $(a*b*c)$ the sum over $\alpha,\beta,\gamma$ of
the $a$-th summand in the first line,
the $b$-th summand in the second, and the $c$-th summand of third line of~\eqref{eq:threetimesthree}.
Then $(3*3*3)$ obviously yields, after summation over $X,Y,U,V,S,T\in\{X_1,\ldots,X_m\}$, a multiple
of $I_{\alpha\gamma|\beta\gamma|\alpha\beta}$, and thus contributes only to~$L_2$.
Five of the other $(a*b*c)$ (for example, $(1*1*1)$)
lead to multiples of certain $I_{s_1\ldots s_6}$ in which the only subtuple is of length six
(the reason being, just as we noted in the proof of~(i),
that there is no way to group the six factors into subsets which
would not be linked to each other by the occurrence any common vectors from $\{X,Y,U,V,S,T\}$); these
again contribute only to~$L_2$.
The only products which instead lead to a multiple of $I_{\alpha\beta\gamma|\alpha\beta\gamma}$
are $(1*1*2)$, $(1*2*1)$, $(2*1*1)$, and $(2*2*2)$. For example,
$$(1*1*2)=-\tfrac1{64}\tsum_{\alpha,\beta,\gamma}
  \<j_{Z_\alpha}Y,U\>\<j_{Z_\beta}U,T\>\<j_{Z_\gamma}T,Y\>
  \<j_{Z_\alpha}X,V\>\<j_{Z_\beta}V,S\>\<j_{Z_\gamma}S,X\>\\
$$
which after summation gives
$-\tfrac1{64}\tsum_{\alpha,\beta,\gamma}\Tr(j_{Z_\alpha}j_{Z_\beta}j_{Z_\gamma})\Tr(j_{Z_\alpha}j_{Z_\beta}
 j_{Z_\gamma})=-\tfrac1{64}I_{\alpha\beta\gamma|\alpha\beta\gamma}$.
The result is the same for each of the three other products just mentioned. So we obtain
$$4\cdot(-\tfrac1{64})I_{\alpha\beta\gamma|\alpha\beta\gamma}=-\tfrac1{16}I_{\alpha\beta\gamma|\alpha\beta\gamma}
$$
as the contribution to~$\hat R$ of the first summand in~\eqref{eq:twopartsrhat}, apart from
its contributions to~$L_2$.

For the second summand in~\eqref{eq:twopartsrhat}, we compute
\begin{equation}
  \begin{split}
  \<R(X,Y)U,&V\>\<R(U,V)Z,W\>\<R(Z,W)X,Y\>=\\
  \tsum_{\alpha=1}^r\hphantom{\cdot}
    &\bigl(\tfrac14\<j_{Z_\alpha}Y,U\>\<j_{Z_\alpha}X,V\>
    -\tfrac14\<j_{Z_\alpha}X,U\>\<j_{Z_\alpha}Y,V\>
    -\tfrac12\<j_{Z_\alpha}X,Y\>\<j_{Z_\alpha}U,V\>\bigr)\\
  \cdot&\tfrac1{16}\<[j_Z,j_W]U,V\>\<[j_Z,j_W]X,Y\>.
\end{split}
\end{equation}
The first two summands from the first line, multiplied with the factors from the second line,
will, after summation, yields multiples of certain $I_{s_1\ldots s_6}$ in which the only
subtuple is of length six; this gives a contribution to~$L_2$. The remaining term is
$$-\tfrac1{32}\tsum_\alpha\<j_{Z_\alpha}X,Y\>\<j_{Z_\alpha}U,V\>\<[j_Z,j_W]U,V\>\<[j_Z,j_W]X,Y\>
$$
which after summation over $X,Y,U,V$ gives $-\tfrac1{32}\sum_\alpha(\Tr(j_{\Z_\alpha}[j_Z,j_W]))^2$;
using skew-symmetry of the maps involved, this simplifies to $-\tfrac18\sum_\alpha(\Tr(j_{Z_\alpha}j_Zj_W))^2$.
Summation over $Z,W\in\{Z_1,\ldots,Z_r\}$ thus gives $-\tfrac18I_{\alpha\beta\gamma|\alpha\beta\gamma}$.
Hence, we obtain
$$3\cdot(-\tfrac18)I_{\alpha\beta\gamma|\alpha\beta\gamma}
$$
as the contribution to~$\hat R$ of the second summand in~\eqref{eq:twopartsrhat}, apart from
its contributions to~$L_2$. The statement now follows by $-\frac1{16}-\frac38=-\frac7{16}$.

(iii)
Although it would be possible to prove (iii) directly, similarly to the above proofs
for (i) and~(ii), we prefer to use the results of (i), (ii) together with those
from Lemma~\ref{lem:nablaric} and the integral relation from Proposition~\ref{prop:integralrelations}(iii).
If $G(j)$ admits a compact quotient, then it follows from local homogeneity and
Proposition~\ref{prop:integralrelations}(iii)
that
$$\rcirc=-|\nabla\ric|^2+\tfrac14|\nabla R|^2-\Tr(\Ric^3)+(*)+\tfrac12(**)-\tfrac14\hat R.
$$
By (i), (ii) and Lemma~\ref{lem:nablaric}, the right hand side is indeed of the form
$$\tfrac14\cdot(-\tfrac32)I_{\alpha\beta\gamma|\alpha\beta\gamma}
  -\tfrac14\cdot(-\tfrac7{16})I_{\alpha\beta\gamma|\alpha\beta\gamma}+L_3
  =-\tfrac{17}{64}I_{\alpha\beta\gamma|\alpha\beta\gamma}+L_3,
$$
where $L_3$ is a linear
combination of invariants in which only subtuples of even length occur. So we have proved
the statement of~(iii) in the case that $G(j)$ admits a compact quotient.

The statement in the general case now follows by continuity. In fact,
any $G(j)$ for which $j$ which is a rational map w.r.t.~the standard rational structures
on $\z=\R^r$ and $\so(\v)=\so(m)$ does admit a compact quotient, and the rational maps
are dense in the space of all linear maps $j:\z\to\so(\v)$.
\end{proof}

\end{document}